\documentclass[10pt,A4paper,fleqn]{amsart}
\usepackage{mathrsfs}
\usepackage{amsfonts}
\usepackage{mathbbold}
\usepackage{txfonts}
\usepackage{amsmath}
\usepackage{amssymb}
\usepackage{amsthm}
\usepackage[pdftex]{graphicx}
\usepackage[toc,page,title,titletoc,header]{appendix}
\usepackage{geometry}
\usepackage{upgreek}
\usepackage[T1]{fontenc}
\usepackage{color}
\usepackage{hyperref}
\usepackage[all]{xy}
\usepackage[french,english]{babel}

\theoremstyle{plain}
\newtheorem{theo}{Theorem}[section]
\newtheorem*{theo*}{Theorem}
\newtheorem{prop}{Proposition}[section]

\newtheorem{cor}{Corollary}[section]
\newtheorem*{lem*}{Lemma}

\theoremstyle{definition}
\newtheorem{lem}{Lemma}[section]

\newtheorem{condition}{Condition}[section]
\newtheorem{example-condition}{Example-Condition}[section]

\theoremstyle{remark}

\newtheorem*{ack}{Acknowledgements}
\newtheorem*{rem*}{Remark}
\newtheorem{rem}{Remark}[section]

\newcommand{\R}{\mathbb{R}}
\newcommand{\T}{\mathbb{T}}
\newcommand{\Z}{\mathbb{Z}}
\newcommand{\C}{\mathbb{C}}

\newcommand{\N}{\mathbb{N}}

\begin{document}

\bibliographystyle{unsrt}

\abstract
In a system of particles, quasi-periodic almost-collision orbits are collisionless orbits along which two bodies become arbitrarily close to each other --  the lower limit of their distance is zero but the upper limit is strictly positive -- and which are quasi-periodic in a regularized system up to a change of time. The existence of such orbits was shown in the restricted planar circular three-body problem by A. Chenciner and J. Llibre, and later, in the planar three-body problem by J. F\'ejoz. In the spatial three-body problem, the existence of a set of positive measure of such orbits was predicted by C. Marchal. In this article, we present a proof of this fact.
\endabstract

\title{Quasi-periodic Almost-collision orbits in the Spatial Three-Body Problem}

\date\today
\author{Lei Zhao}
\address{ASD, IMCCE, Observatoire de Paris, 77 Avenue Denfert-Rochereau 75014 Paris, France / Université Paris Diderot }
\email{zhaolei@imcce.fr}
\maketitle

\tableofcontents

\section{Introduction}
\subsection{Quasi-periodic Almost-collision Orbits}
In Chazy's classification of final motions of the three-body problem (see \cite[P.~83]{ArnoldEncyclopedia}), possible final velocities were not specified for two particular kinds of possible motions: bounded motions, \emph{i.e.} those motions such that the mutual distances remain bounded when time goes to infinity, and oscillating motions, \emph{i.e.} those motions along which the upper limit of the mutual distances goes to infinity, while the lower limit of the mutual distances remains finite. A number of bounded motions and a few oscillating motions were known, with Sitnikov's model being one of the well-known examples of the latter kind. 

If we replace the oscillation of mutual distances by the oscillation of relative velocities of the bodies, we obtain another kind of oscillating motions, which was called by C. Marchal ``oscillating motions of the second kind''. By consulting the criteria for velocities in Chazy's classification, we see that if such motions do exist and are not (usual) oscillating motions, then they must be bounded.

In \cite{MarchalCollision}, by analyzing an integrable approximating system of the lunar spatial three-body problem (in which a far away third body is added to a two-body system) near a degenerate inner ellipse, C. Marchal became aware of the existence of a positive measure of such motions in the spatial three-body problem. More precisely, the predicted motions
\begin{itemize}
\item are with incommensurable frequencies;
\item arise from invariant tori of the quadrupolar system $F_{sec}^{1,2}$ (see Subsection \ref{Subsection: Elimination of g_{2}} for its definition);
\item  form a possibly nowhere dense set with small but positive measure in the phase space.
\end{itemize}

In this article, we shall investigate a particular kind of oscillating motions of the second kind of the spatial three-body problem: the \emph{quasi-periodic almost-collision orbits}, which are, by definition, collisionless orbits along which two bodies get arbitrarily close to each other: the lower limit of their distance is zero but the upper limit is strictly positive, and which are quasi-periodic in a regularized system. Moreover, we shall show the existence of a set of positive measure of such orbits arising from the (regularized) quadrupolar invariant tori. These are exactly the orbits predicted by C. Marchal.  
As noted by him, since real bodies occupy positive volumes in the universe, the existence of a set of positive measure of almost-collision quasi-periodic motions implies a positive probability of collisions in triple star systems {with one body far away from the other two, and the probability is uniform with respect to their (positive) volumes}. The collision mechanism given by quasi-periodic almost-collision orbits has much larger probability, and is thus more important than the mechanism given by direct collisions in the particle model, {especially when the volumes of the modeled real massive bodies are small}.

The first rigorous mathematical study of quasi-periodic almost-collision orbits was achieved by A. Chenciner and J. Llibre in \cite{ChencinerLlibre}, where they considered the planar circular restricted three-body problem in a rotating frame with a large enough Jacobi constant which determines a Hill region with three connected components. After regularizing the dynamics near the double collision of the astroid with one of the primaries by Levi-Civita regularization, they reduced the dynamical study to the study of the corresponding Poincar{\'e} map on a global annulus of section in the regularized system. This Poincar{\'e} map is a twist map with a small twist perturbed by a much smaller perturbation, which makes it possible to apply Moser's invariant curve theorem to establish the persistence of a set of positive measure of invariant KAM tori. By adjusting the Jacobi constant, a set of positive measure of such invariant tori was shown to intersect transversally the codimension 2 \emph{collision set} (the set in the regularized phase space corresponding to the double collision of the astroid with the primary). Such invariant tori were called invariant ``punctured'' tori, as in the (non-regularized) phase space, they have a finite number of punctures corresponding to collisions. As the flow is linear and ergodic on each  KAM torus, most of the orbits will not pass through but will get arbitrary close to the collision set. These orbits give rise to a set of positive measure of quasi-periodic almost-collision orbits in the planar circular restricted three-body problem. 

In his thesis \cite{Fthesis} (and in the article \cite{QuasiMotionPlanar}), J. F{\'e}joz generalized the study of Chenciner-Llibre to the planar three-body problem. In his study, the inner double collisions being regularized, the \emph{secular regularized systems}, \emph{i.e.} the normal forms one gets by averaging over the fast angles, are established with the same averaging method as the usual non-regularized ones. A careful analysis shows that the dynamics of the secular regularized system and the naturally extended (through degenerate inner ellipses) secular systems are orbitally conjugate, up to a modification of the mass of the third body which is far away from the inner pair. The persistence of a set of positive measure of invariant tori is obtained by application of a sophisticated version of KAM theorem. After verifying the transversality of the intersections between the KAM tori and the codimension 2 collision set corresponding to collisions of the inner pair, he concluded in the same way as Chenciner-Llibre. 

In this article, we generalize the studies of Chenciner-Llibre and F\'ejoz to the spatial three-body problem, and hence confirm the prediction of C. Marchal. We prove
\begin{theo}\label{Theo: Main}In the spatial three-body problem, there exists a set of positive measure of quasi-periodic almost-collision orbits on each negative energy surface. It follows that the set of quasi-periodic almost-collision orbits has positive measure in the phase space.
\end{theo}

\subsection{Outline of the proof}
We assume that we are in the ``lunar case'', that is the case of a pair of two bodies far away from the third one. We decompose the Hamiltonian $F$ of the three-body problem into two parts
$$F=F_{Kep} + F_{pert},$$
where $F_{Kep}$ is the sum of two uncoupled Keplerian Hamiltonians, and $F_{pert}$ is significantly smaller than each of the Keplerian Hamiltonians in $F_{Kep}$. The dynamics of $F$ can thus be described as {almost} uncoupled Keplerian motions with slow evolutions of the Keplerian orbits (see Section \ref{Section: Formulation}).

As in \cite{ChencinerLlibre}, \cite{QuasiMotionPlanar}, the strategy is to find a set of positive measure of irrational tori in the corresponding energy level of a regularized system $\mathcal{F}$ of $F$. More precisely, we shall 
\begin{enumerate}
\item regularize the inner double collisions of $F$ on the energy surface $F=-f<0$ by  Kustaanheimo-Stiefel regularization to obtain a Hamiltonian $\mathcal{F}$ regular at the \emph{collision set}, corresponding to inner double collisions of $F$ (see Section \ref{Section: Regularization});
\item build an integrable truncated normal form $\mathcal{F}_{Kep}+\overline{\mathcal{F}_{sec}^{n, n'}}$ of $\mathcal{F}$ (see Section \ref{Section: Normal Forms}) and study its Lagrangian invariant tori passing near the collision set (see Section \ref{Section: Dynamics of the Integrable Normal Forms}, where in particular, we extend Lidov-Ziglin's study \cite{LidovZiglin} of the quadrupolar dynamics to degenerate inner ellipses); 
\item apply an iso-energetic proper-degenerate KAM theorem to find a set of positive measure of invariant tori of $\mathcal{F}$ on its zero-energy level (which is the only one on which the dynamics of $\mathcal{F}$ extends the dynamics of $F$, see Section \ref{Section: Application of KAM theorem});
\item show that a set of positive measure of invariant ergodic tori intersect transversely the collision set in submanifolds of codimension at least 2; conclude that there exists a set of positive measure of quasi-periodic almost-collision orbits on the energy surface $F=-f$; finally, by varying $f$, conclude that these orbits form a set of positive measure in the phase space of $F$ (see Section \ref{Section: Transversality}).
\end{enumerate}

\section{Hamiltonian formalism of the three-body problem} \label{Section: Formulation}
\subsection{The Hamiltonian}
The three-body problem is a Hamiltonian system with phase space
$$\Pi:=\left\{(p_{j}, q_{j})_{j=0,1,2}=(p_{j}^{1}, p_{j}^{2}, p_{j}^{3}, q_{j}^{1}, q_{j}^{2}, q_{j}^{3}) \in (\R^{3} \times \R^3)^3 |\,  \forall 0 \leq j \neq k \leq 2, q_j \neq q_k \right\}, $$
(standard) symplectic form 
$$\sum^{2}_{j=0} \sum^{3}_{l=1} d p_j^l \wedge d q_j^l,$$
and the Hamiltonian function
$$F=\dfrac{1}{2} \sum_{0 \le j \le 2} \dfrac{\|p_j\|^2}{m_j} -  \sum_{0 \le j < k \le 2} \dfrac{m_j m_k}{\|q_j- q_k\|},$$
in which $q_0,q_1,q_2$ denote the positions of the three particles, and $p_0,p_1,p_2$ denote their conjugate momenta respectively. The Euclidean norm of a vector in $\R^{3}$ is denoted by $\|\cdot\|$. The gravitational constant has been set to $1$.

\subsection{Jacobi decomposition}

The Hamiltonian $F$ is invariant under translations in positions. To reduce the system by this symmetry, we {switch} to the \emph{Jacobi coordinates} $(P_i, Q_i),i=0, 1, 2,$ defined as
\begin{equation*} 
\left\{
\begin{array}{l} P_0=p_0+p_1+ p_2 \\ P_1=p_1+ \sigma_1 p_2\\ P_2 = p_2
\end{array} \right.
\hbox{ \phantom{aaaaaaaqqqqaa}}
\left\{
\begin{array}{l} Q_0=q_0 \\ Q_1=q_1- q_0 \\ Q_2=q_2-\sigma_0 q_0-\sigma_1 q_1,
\end{array} \right.
\end{equation*}
where 
$$\dfrac{1}{\sigma_0}=1+\dfrac{m_1}{m_0},  \dfrac{1}{\sigma_1}=1+\dfrac{m_0}{m_1}.$$
Due to the symmetry, the Hamiltonian function is independent of $Q_{0}$. We fix the first integral $P_{0}$ (conjugate to $Q_{0}$) at $P_{0}=0$ and reduce the translation symmetry of the system by eliminate $Q_{0}$. In coordinates $(P_i, Q_i),i=1, 2$, the (reduced) Hamiltonian function $F=F(P_{1}, Q_{1}, P_{2}, Q_{2})$ thus describes the motion of two fictitious particles. 

We further decompose the Hamiltonian $F(P_{1}, Q_{1}, P_{2}, Q_{2})$ into two parts $$F=F_{Kep}+F_{pert},$$ where the \emph{Keplerian part} $F_{Kep}$ and the \emph{perturbing part} $F_{pert}$ are respectively
\begin{align*}
&F_{Kep}=\dfrac{\|P_1\|^2}{2 \mu_1}+\dfrac{\|P_2\|^2}{2 \mu_2}-\dfrac{\mu_1 M_1}{\|Q_1\|}-\dfrac{\mu_2 M_2}{\|Q_2\|},
  \\&F_{pert}=-\mu_1 m_2\Bigl[\dfrac{1}{\sigma_o}\bigl(\dfrac{1}{\|Q_2-\sigma_0 Q_1\|}-\dfrac{1}{\|Q_2\|}\bigr)+\dfrac{1}{\sigma_1}\bigl(\dfrac{1}{\|Q_2+\sigma_1 Q_1\|}-\dfrac{1}{\|Q_2\|}\bigr)\,\Bigr],
 \end{align*} \label{Not: pert part}
 with (as in \cite{QuasiMotionPlanar})
 \begin{align*}&\dfrac{1}{\mu_1}=\dfrac{1}{m_0}+\dfrac{1}{m_1}, \, \dfrac{1}{\mu_2}=\dfrac{1}{m_0+m_1}+\dfrac{1}{m_2},\\ & M_1=m_0+m_1, M_2=m_0+m_1+m_2.
 \end{align*} 

We shall only be interested in the region of the phase space where $F=F_{Kep}+F_{pert}$ is a small perturbation of a pair of Keplerian elliptic motions.

\section{Regularization} \label{Section: Regularization}
We aim at carrying a perturbative study near the inner double collisions $\|Q_{1}\|=0$ where the Hamiltonian $F$ is singular. To this end, we have to regularize the system. We shall use Kustaanheimo-Stiefel regularization (c.f. \cite{SS}) and, starting with a formula appearing in \cite{Mikkola}, we formulate this method in a quaternionic way (See also \cite{KSregularization}). Another slightly different quaternionic formulation can be found in \cite{Waldvogel}.

\subsection{Kustaanheimo-Stiefel regularization}
Let $\mathbb{H} (\cong \R^{4})$ be the space of quaternions $z=z_{0}+z_{1} i+ z_{2} j + z_{3} k$, $\mathbb{IH} (\cong \R^{3})$ be the space of purely imaginary quaternions ($z_{0}=0$), $\bar{z}=z_{0}-z_{1} i- z_{2} j - z_{3} k$ and $|z|=\sqrt{\bar{z} z}$ be the conjugate quaternion and the quaternionic module of $z$ respectively.

Let $(z, w) \in \mathbb{H} \times \mathbb{H} \cong T^{*} \mathbb{H} $ be a pair of quaternions. 
We define the Hopf mapping
\begin{align*}
         \mathbb{H} \setminus \{0\}  \rightarrow \mathbb{IH} \setminus \{0\} \\
        \qquad \phantom{\mathbb{H}} z\,\, \mapsto Q=\bar{z} i z
\end{align*}
and the Kustaanheimo-Stiefel mapping  
\begin{align*}
  K.S.: &T^*(\mathbb{H} \setminus \{0\})   \rightarrow \mathbb{IH} \times \mathbb{H} \\
        &(z,w) \longmapsto (Q=\bar{z} i z,P=\dfrac{ \bar{z} i w}{2{|z|}^2}).
\end{align*}

We observe that both mappings have respectively circle fibres $\{e^{i \vartheta} z\}$ and $\{(e^{i \vartheta} z, e^{i \vartheta} w)\}$ for $\vartheta \in \R/ 2 \pi \Z$. The fibres of $K.S.$ define a Hamiltonian circle action on $(T^*\mathbb{H}, Re\{ d \bar{w} \wedge d z\})$ with (up to sign) the moment map  $$BL(z,w)= Re\{\bar{z} i w\}:=\bar{z} i w + \overline{\bar{z} i w}$$
 regarded as a function defined on $T^{*} \mathbb{H} \cong \mathbb{H} \times \mathbb{H}$. The equation
$$BL(z, w)=0$$
thus defines a 7-dimensional quadratic cone $\Sigma$ in $T^{*} \mathbb{H}$. By removing the point $(0,0)$ from $\Sigma$, we obtain a 7-dimensional coisotropic submanifold $\Sigma^{0}$ of the symplectic manifold $(T^*\mathbb{H}, Re\{ d \bar{w} \wedge d z\})$, on which the above-mentioned circle action is free. The quotient $V^{0}$ of $\Sigma^{0}$ by this action is thus a 6-dimensional symplectic manifold equipped with the induced symplectic form $\omega_{1}$. 

We define the 7-dimensional coisotropic submanifold $\Sigma^{1}$ by removing $\{z=0\}$ from $\Sigma$. Analogously, by passing to the quotient, it descends to a 6-dimensional symplectic manifold $(V^{1}, \omega_{1})$.

\begin{prop}\cite[Proposition 3.3]{KSregularization} $K.S.$ induces a symplectomorphism from $(V^{1},\omega_1)$ \\ \phantom{aaaaaaaaaaaaaaa} to $\left(T^*(\mathbb{IH}\setminus \{0\}),\,Re\{d \bar{P} \wedge d Q\}\right)$.
\end{prop} 

\subsection{Regularized Hamiltonian}
On the fixed negative energy surface 
$$F=-f < 0,$$
we make a time change (singular at inner double collisions) by passing to the new time variable $\tau$ satisfying 
$$\|Q_{1}\| \,d \tau = d t.$$ 
In time $\tau$, the corresponding motions of the particles are governed by the Hamiltonian $\|Q_{1}\| (F+f)$ and are lying inside its zero-energy level. We extend $K.S.$ to the mapping ({the notation $K. S.$ is abusively maintained for the extension})
$$(z, w, P_{2}, Q_{2}) \mapsto (Q_{1}=\bar{z} i z, P_{1}=\dfrac{\bar{z} i w}{2 |z|^{2}}, P_{2}, Q_{2})$$
and set
$$\mathcal{F}=K.S.^{*} \Bigl(\|Q_{1}\|\, (F+f)\Bigr).$$
This is a function on $\Sigma^{0} \times T^{*} (\R^{3} \setminus \{0\})$ decomposed as
$$\mathcal{F}=\mathcal{F}_{Kep}+\mathcal{F}_{pert},$$
with the \emph{regularized Keplerian part}
$$\mathcal{F}_{Kep} = K.S.^*\Bigl(\|Q_1\|(F_{Kep} + f)\Bigr)=\dfrac{|w|^2}{8 \mu_1}\, +\Bigl(f+\dfrac{\|P_2\|^2}{2 \mu_2}-\dfrac{\mu_2 M_2}{\|Q_2\|}\Bigr)|z|^2\,-\mu_1 M_1$$
describing the skew-product motion of the outer body moving on an Keplerian elliptic orbit, slowed-down by four ``inner'' harmonic oscillators in $1:1:1:1$-resonance,
and the \emph{regularized perturbing part}
$$\mathcal{F}_{pert} = K.S.^*\Bigl(\|Q_1\|F_{pert}\Bigr);$$
both terms extend analytically through the set $\{z=0\}$ corresponding to inner double collisions of $F$. 

By its expression, the function $\mathcal{F}$ can be directly regarded as a function on $T^{*} (\mathbb{H} \setminus \{0\}) \times T^{*} (\R^{3}\setminus \{0\})$. As it is invariant under the fibre action of $K. S.$, its flow preserves $\Sigma^{0} \times T^{*} (\R^{3} \setminus \{0\})$. In the sequel, the relation $BL(z, w)=0$ is always assumed, \emph{i.e.} we always restrict to $\Sigma^{0} \times T^{*} (\R^{3} \setminus \{0\})$.

\subsection{Regular Coordinates} \label{Subsection: Regular Coordinates}
To write $F_{Kep}$ in action-angle form, we start by defining the ({symplectic}) Delaunay coordinates $(L_2,l_2,G_2,g_2,H_2,h_2)$ for the outer body. Let $a_{2}, e_{2}, i_{2}$ be respectively the semi major axis, the eccentricity and the inclination of the outer ellipse. The Delaunay coordinates are defined as the follows:

\begin{equation*} 
\left\{
\begin{array}{ll}L_2=\mu_2 \sqrt{M_2} \sqrt{a_2}   & \hbox{circular angular momentum}\\ l_2  &\hbox{mean anomaly}\\ G_2 = L_2 \sqrt{1-e_2^2} &\hbox{angular momentum} \\g_2 &\hbox{argument of pericentre} \\ H_2=G_2 \cos i_2 &\hbox{vertical component of the angular momentum} \\ h_2 &\hbox{ longitude of the ascending node}.
\end{array}\right.
\end{equation*}

In these coordinates, we have 
$$f+\dfrac{\|P_2\|^2}{2 \mu_2}-\dfrac{\mu_2 M_2}{\|Q_2\|}=f-\dfrac{\mu_{2}^{3} M_{2}^{2}}{2 L_{2}^{2}},$$
which is positive by hypothesis. {We denote it by $f_{1} (L_{2})$.} Now
$$\mathcal{F}_{Kep} = \dfrac{|w|^2}{8 \mu_1}\, + f_{1} (L_{2}) |z|^2\,-\mu_1 M_1.$$

Let  
$$
\sqrt{8 \mu_{0} f_{1} (L_{2})}  \, z_i=\sqrt{2 I_i} \sin{\phi_i}, 
\quad w_i=\sqrt{2 I_i} \cos{\phi_i}, \quad i=0,1,2,3.
$$
and 
\begin{align*}\phantom{ccccccccccccccc}
&\mathcal{P}_0=\dfrac{(I_0+I_1+I_2+I_3)}{2 \sqrt{8 \mu_{0} f_{1} (L_{2})}}, \, \vartheta_0= 2 \phi_0,\\
&\mathcal{P}_i=\dfrac{I_i}{\sqrt{8 \mu_{0} f_{1} (L_{2})}},  \qquad \quad \vartheta_i=\phi_i-\phi_0, \qquad i=1,2,3.
\end{align*}
One directly checks that 
\begin{align*}
&\mathcal{P}_0 \wedge \vartheta_0 + \mathcal{P}_1 \wedge \vartheta_1 +\mathcal{P}_2 \wedge \vartheta_2 + \mathcal{P}_3 \wedge \vartheta_3 + d L_{2} \wedge d l_{2}' + d G_{2} \wedge d g_{2} + d H_{2} \wedge d h_{2}\\=& Re\{d \bar{w} \wedge d z\} + d L_{2} \wedge d l_{2} + d G_{2} \wedge d g_{2} + d H_{2} \wedge d h_{2},
\end{align*}
with $l_{2}'=l_{2} + \dfrac{f_{1}'(L_{2})}{2 f_{1}(L_{2})} Re\{\bar{P}_{1} Q_{1}\}$. 

We thus obtain a set of Darboux coordinates (we call them \emph{regular coordinates})
$$(\mathcal{P}_0, \vartheta_{0}, \mathcal{P}_{1}, \vartheta_{1}, \mathcal{P}_2, \vartheta_{2}, \mathcal{P}_{3}, \vartheta_{3}, L_{2}, l_{2}', G_{2}, g_{2}, H_{2}, h_{2}),$$
in which
$$\mathcal{F}_{Kep}=\mathcal{P}_{0} \sqrt{\dfrac{2 f_{1} (L_{2})}{\mu_{1}}}-\mu_{1} M_{1}. $$

The coordinates $(\mathcal{P}_0, \vartheta_{0}, \mathcal{P}_{1}, \vartheta_{1}, \mathcal{P}_2, \vartheta_{2}, \mathcal{P}_{3}, \vartheta_{3})$ are well-defined on the dense open set of\\ $T^{*} \mathbb{H}\setminus\{(0,0)\}$, defined by 
$$I_{i}>0, i=0,1,2,3,$$
where the projections of the elliptic orbit of the four harmonic oscillators in $1:1:1:1$ resonance in the four $(z_{i}, w_{i})$ planes are non-degenerate, a condition which can always be satisfied by simultaneously rotating these planes properly.
In the sequel, without loss of generality, we shall always assume that these conditions are satisfied.

\section{Normal Forms}\label{Section: Normal Forms}
\subsection{Physical dynamics of $\mathcal{F}_{Kep}$}
The function $\mathcal{F}_{Kep}$ describes a \emph{properly degenerate} Hamiltonian system: it depends only on 2 of the action variables out of 7. To deduce the dynamics of $\mathcal{F}$, study of higher order is thus necessary. 

The perturbing part $\mathcal{F}_{pert}$ describes the mutual interaction between two particles $Q_{1}= \bar{z} i z$ and $Q_{2}$ in the physical space. Under $\mathcal{F}_{Kep}$, the particle $Q_{2}$ moves on elliptic orbits. When the energy of $\mathcal{F}_{Kep}$ is close to zero, this is also the case for $Q_{1}$:

\begin{lem}\label{lem: physical ellipse}
Under the flow of $\mathcal{F}_{Kep}$, when $BL(z, w)=0$, in the energy hypersurface $\mathcal{F}_{Kep}=\tilde{f}$ for any $\tilde{f} > -\mu_{1} M_{1}$, the physical image $Q_{1}= \bar{z} i z$ of $z$ moves on a Keplerian elliptic orbit.
\end{lem}
\begin{proof}
The equation
$$\mathcal{F}_{Kep}=\dfrac{|w|^2}{8 \mu_1}\, +f_{1} (L_{2}) \,|z|^2\,-\mu_1 M_1=\tilde{f} $$
is equivalent to
 $$\|Q_{1}\|\Bigl(\dfrac{\|P_{1}\|^2}{2 \mu_{1}} - \dfrac{\mu_{1} M_{1}+\tilde{f} }{\|Q_{1}\|}+f_{1} (L_{2})\Bigr)=0,$$
 that is
 $$\dfrac{\|P_{1}\|^2}{2 \mu_{1}}-\dfrac{\mu_{1} M_{1}+\tilde{f}}{\|Q_{1}\|}=-f_{1} (L_{2}) < 0.$$
By assumption $\mu_{1} M_{1} + \tilde{f} >0$. The motion of $Q_{1}$ is thus governed, up to time parametrization, by the Hamiltonian of a Kepler problem on a fixed negative energy surface; as the orbits are uniquely determined by their energy, the conclusion follows. 
\end{proof}

We have seen from the above proof that, in the physical space, inner Keplerian ellipses are orbits of the Kepler problem with (modified) mass parameters $\mu_{1}$ and $M_{1}+\dfrac{\mathcal{F}_{Kep}}{\mu_{1}}$. The inner elliptic elements, \emph{e.g.} the inner semi major axis $a_{1}$ and the inner eccentricity $e_{1}$, are the corresponding elliptic elements of the orbit of $Q_{1}$. One directly checks that $a_{1}=\dfrac{\mathcal{P}_{0}}{\sqrt{2 \mu_{1} f_{1} (L_{2}) }}$, and $\vartheta_{0}$ is an eccentric longitude of the inner motion, which differs from its eccentric anomaly $u_{1}$ only by a phase shift. 

\subsection{Asynchronous elimination} \label{Subsection: Asynchronous elimination}
Let $e_{1}, e_{2}$ be the eccentricities of the inner and outer ellipses respectively, $a_{2}$ be the outer semi major axis, and $\alpha=\dfrac{a_{1}}{a_{2}}$ be the ratio of semi major axes. 
We assume that
\begin{itemize}
\item the masses $m_{0}, m_{1}, m_{2}$ are (arbitrarily) fixed;
\item the coordinates $(\mathcal{P}_{0}, \vartheta_{0}, \mathcal{P}_{1}, \vartheta_{1}, \mathcal{P}_{2}, \vartheta_{2}, \mathcal{P}_{3}, \vartheta_{3}, L_{2}, l_{2}, G_{2}, g_{2}, H_{2},h_{2})$ are all well-defined;
\item three positive real numbers $e_{1}^{\vee}, e_{2}^{\vee} < e_{2}^{\wedge}$ are fixed and
$$0<e_{1}^{\vee} < e_{1} \le1,\quad 0<e_{2}^{\vee} < e_{2} < e_{2}^{\wedge}<1;$$
\item two positive real numbers $a_{1}^{\vee} < a_{1}^{\wedge}$ are fixed, and
$$a_{1}^{\vee} < a_{1} < a_{1}^{\wedge};$$
\item $\alpha < \alpha^{\wedge}:=\min\{\dfrac{1-e_{2}^{\wedge}}{80}, \dfrac{1-e_{2}^{\wedge}}{2 \sigma_{0}}, \dfrac{1-e_{2}^{\wedge}}{2 \sigma_{1}}\};$
\end{itemize}
We shall take $\alpha$ as the small parameter in this study.

These assumptions determine a subset $\mathcal{P}^{*}$ of $T^{*} (\mathbb{H} \setminus \{0\}) \times T^{*} (\R^{3} \setminus \{0\})$. Using the coordinates defined above, we may identify $\mathcal{P}^{*}$ to a subset of $\T^{7} \times \R^{7}$.

With these assumptions, $|\mathcal{F}_{Kep}|$ is bounded, and the two Keplerian frequencies 
$$\nu_{1}=\dfrac{\partial \mathcal{F}_{Kep}}{\partial P_{0}}=\sqrt{\dfrac{2 f_{1} (L_{2})}{\mu_{1}}} \sim 1, \,\,\nu_{2}=\dfrac{\partial \mathcal{F}_{Kep}}{\partial L_{2}}= \dfrac{\mu_{2}^{3} M_{2}^{2} \mathcal{P}_{0} }{2 L_{2}^{3} \sqrt{2 \mu_{1} f_{1}(L_{2})}} \sim \alpha^{\frac{3}{2}}$$
do not appear at the same order of $\alpha$. This enables us to proceed, as in Jefferys-Moser \cite{JefferysMoser} or F\'ejoz \cite{QuasiMotionPlanar}, by eliminating the dependence of $\mathcal{F}$ on each of the fast angles $\vartheta_{0}, l_{2}'$ without imposing any arithmetic condition on the two frequencies. 

Let $T_{\C}=\C^7/\Z^7 \times \C^7$ and $T_s$ be the $s$-neighborhood of $\T^7 \times \R^7:=\R^{7}/\Z^{7} \times \R^{7}$ in $T_{\C}$. Let $T_{\textbf{A},s}$ be the s-neighborhood $\{z \in T_{s}: \exists x \in \textbf{A}, \hbox{s.t. } |z-x| < s\}$ of a set $\textbf{A} \subset \T^7 \times \R^7$ in $T_s$. The complex modulus of a transformation is the maximum of the complex moduli of its components. We use $| \cdot |$ to denote the modulus of either one. A real analytic function and its complex extension are denoted by the same notation.

\begin{prop}\label{prop: fast angle averaging} For any $n \in \N$, there exists an analytic Hamiltonian $\mathcal{F}^n: \mathcal{P}^{*} \to \R$ independent of the fast angles $\vartheta_1, l_{2}'$, and an analytic symplectomorphism $\phi^n: \mathcal{P}^{*} \supset \tilde{\mathcal{P}}^{n} \to \phi^n (\tilde{\mathcal{P}}^{n})$, $|\alpha|^{\frac{3}{2}}$-close to the identity, such that
             $$|\mathcal{F} \circ \phi^n- \mathcal{F}^n| \le C_{0}\, |\alpha |^{\frac{3(n+2)}{2}}$$
on $T_{\tilde{\mathcal{P}}^{n}, s^{''}}$ for some open set $\tilde{\mathcal{P}}^{n} \subset \mathcal{P}^{*}$, and some real number $s^{''}$ with $0< s^{''} < s$, such that locally the relative measure of $\tilde{\mathcal{P}}^{n}$ in $\mathcal{P}^{*}$ tends to 1 when $\alpha$ tends to $0$. 
\end{prop}

\begin{proof} We first eliminate the dependence of $\vartheta_{1}$ in $\mathcal{F}$ up to a remainder of order $|\alpha |^{\frac{3(n+2)}{2}}$ and then eliminate the dependence of $l_{2}'$ up to a remainder of order $|\alpha |^{\frac{3(n+2)}{2}}$. 

The elimination procedure is standard and consists in analogous successive steps.
The first step is to eliminate $\vartheta_{1}$ up to a $O(\alpha^{\frac{9}{2}})$-remainder. To this end we look for an auxiliary analytic Hamiltonian $\hat{H}$. We denote its Hamiltonian vector field and its flow by $X_{\hat{H}}$ and $\phi_t$ respectively. The required symplectic transformation is the time-1 map $\phi_1(:= \phi_t|_{t=1})$ of $X_{\hat{H}}$. We have
$$\phi_1^*\mathcal{F}=\mathcal{F}_{Kep}+(\mathcal{F}_{pert}+X_{\hat{H}} \cdot \mathcal{F}_{Kep})+\mathcal{F}^1_{comp,1},$$
for some remainder $\mathcal{F}^{1}_{comp, 1}$. In the above, $X_{\hat{H}}$ is seen as a derivation operator. 

Let
$$\langle \mathcal{F}_{pert}\rangle_1 = \dfrac{1}{2 \pi} \int_{0}^{2 \pi} \mathcal{F}_{pert} \,d \vartheta_1$$
be the average of $\mathcal{F}_{pert}$ over $\vartheta_1$, and $\widetilde{F}_{pert,1}=\mathcal{F}_{pert}-\langle \mathcal{F}_{pert}\rangle_1$ be its zero-average part. 

In order to have $\mathcal{F}^{1}_{comp, 1}=O(\alpha^{\frac{9}{2}})$, we choose $\hat{H}$ to solve the equation
$$\nu_1 \partial_{l_1} \hat{H}=\widetilde{\mathcal{F}}_{pert,1},$$
which is satisfied if we set
$$\hat{H}= \dfrac{1}{\nu_1}  \int_0^{l_1}  \widetilde{\mathcal{F}}_{pert,1}\, d l_1,$$
which is of order $O(\alpha^{3})$ in $T_{\mathcal{P}^{*}, s}$ for sufficiently small $s$. By Cauchy inequality, $|X_{\hat{H}}|=O(\alpha^{3})$ in $T_{\mathcal{P}^{*}, s-s_{0}}$ for $0< s_{0} < s/2$.  We shrink the domain from $T_{\mathcal{P}^{*},s- s_0}$ to $T_{\mathcal{P}^{**},s- s_0-s_1}$, where $\mathcal{P}^{**}$ is an open subset of $\mathcal{P}^{*}$, so that $\phi_1(T_{\mathcal{P}^{**},s- s_0-s_1}) \subset T_{\mathcal{P}^{*},s-s_0}$, with $s- s_0-s_1>0$. The time-1 map $\phi_1$ of  $X_H$ thus satisfies $|\phi_1-Id | \le \hbox{Cst } |\alpha |^{3}$ in $T_{\mathcal{P}^{**},s- s_0- s_1}$, and hence
$$\phi^*_1 \mathcal{F} =\mathcal{F}_{Kep}+\langle \mathcal{F}_{pert} \rangle_1 + \mathcal{F}^1_{comp,1},$$
is defined on $T_{\mathcal{P}^{**},s- s_0- s_1}$ and satisfies
 $$\mathcal{F}_{comp,1}^1 = \int_0^1 (1-t) \phi_t^*(X_{\hat{H}}^2 \cdot \mathcal{F}_{Kep})dt+ \int_0^1 \phi_t^*(X_{\hat{H}} \cdot \mathcal{F}_{pert}) dt  - \nu_2 \dfrac{\partial \hat{H}}{\partial l_2} \le \hbox{Cst } |X_{\hat{H}}| (|\widetilde{\mathcal{F}}_{pert,1}|+|\mathcal{F}_{pert}|) + \nu_2 | \hat{H}| \le \hbox{Cst } | \alpha |^{\frac{9}{2}}.$$
The first step of eliminating $\vartheta_{0}$ is completed. 

Analogously, we may eliminate the dependence of the Hamiltonian on $\vartheta_{1}$ up to order $O(\alpha^{\frac{3(n+2)}{2}})$ for any chosen $n \in \Z_{+}$. The Hamiltonian $\mathcal{F}$ is then analytically conjugate to
$$\mathcal{F}_{Kep}+\langle \mathcal{F}_{pert} \rangle_1 + \langle \mathcal{F}^{1}_{comp,1} \rangle_{1}+ \cdots +\langle \mathcal{F}^{1}_{comp,n-1}\rangle_{1} + \mathcal{F}^{1}_{comp,n}, $$
in which the expression $\mathcal{F}_{Kep}+\langle \mathcal{F}_{pert} \rangle_1 + \langle \mathcal{F}^{1}_{comp,1} \rangle_{1}+ \cdots + \langle \mathcal{F}^{1}_{comp,n-1}\rangle_{1}$ is independent of $\vartheta_{1}$, and $\mathcal{F}^{1}_{comp,n}$ is of order $O \left(\alpha^{\frac{3(n+2)}{2}}\right)$. 

The elimination of $l_{2}'$ in $\mathcal{F}_{Kep}+\langle \mathcal{F}_{pert} \rangle_1 + \langle \mathcal{F}^{1}_{comp,1} \rangle_{1}+ \cdots +\langle \mathcal{F}^{1}_{comp,n-1}\rangle_{1} + \mathcal{F}^{1}_{comp,n}$ is analogous. Let 
$\langle \cdot \rangle$ denotes the averaging of a function over $\vartheta_{1}$ and $l_{2}'$. The Hamiltonian generating the transformation of the first step of eliminating $l_{2}'$ is
$$\dfrac{1}{\nu_2} \int_0^{l_2} (\langle F_{pert} \rangle_1- \langle F_{pert} \rangle) d l_2 =O (\alpha^{\frac{3}{2}}).$$
This implies that the transformation is $|\alpha|^{\frac{3}{2}}$-close to the identity.

The Hamiltonian $\mathcal{F}$ is thus conjugate to
$$\mathcal{F}_{Kep}+\langle \mathcal{F}_{pert} \rangle + \langle \mathcal{F}_{comp,1} \rangle+ \cdots +\langle\mathcal{F}_{comp,n-1}\rangle+\mathcal{F}_{comp,n},$$
in which 
$$\langle \mathcal{F}_{pert} \rangle = \dfrac{1}{ 4 \pi^{2}}\int_{0}^{2 \pi} \int_{0}^{2 \pi} \mathcal{F}_{pert} d \vartheta_{1} d l_{2}'. $$
The \emph{$n$-th order secular system}
$$\mathcal{F}_{sec}^{n}:=\langle \mathcal{F}_{pert} \rangle + \langle \mathcal{F}_{comp,1} \rangle+ \cdots +\langle \mathcal{F}_{comp,n-1}\rangle$$
(by convention $\mathcal{F}_{pert}=\mathcal{F}_{comp, 0}$) is independent of $\vartheta_{1}, l_{2}'$, and the {remainder} $\mathcal{F}_{comp,n}$ is of order $O(\alpha^{\frac{3(n+2)}{2}})$ in $T_{\tilde{\mathcal{P}}^{n}, s^{''}}$ for some open subset $\tilde{\mathcal{P}}^{n} \subset \mathcal{P}^{*}$ and some $0<s^{''}<s$ both of which are obtained by a finite number of steps of the elimination procedure. In particular, the set $\tilde{\mathcal{P}}^{n}$ is obtained by shrinking $\mathcal{P}^{*}$ from its boundary by a distance of  $O(\alpha^\frac{3}{2})$. We may thus set 
$$\mathcal{F}^{n}:=\mathcal{F}_{Kep}+\mathcal{F}_{sec}^{n}.$$
\end{proof}

\subsection{Elimination of $g_{2}$}\label{Subsection: Elimination of g_{2}}
The system $\mathcal{F}^{n}$ has 7 degrees of freedom. It is invariant under the action of the group $\T^{3} \times SO(3)$ consisting in the fibre circle action of $K. S.$, the $\T^{2}$-action of the fast angles, and simultaneous rotations of positions and momenta in the phase space.  Standard symplectic reduction procedure leads to a 2 degrees of freedom reduced system with no other obvious symmetries. It is \emph{a priori} not integrable. 

To obtain an integrable approximating system of $\mathcal{F}$, one proceeds classically in the following way: The function $\mathcal{F}_{pert}$ is naturally an analytic function of $a_{1}, a_{2}, \dfrac{Q_{1}}{a_{1}}, \dfrac{Q_{2}}{a_{2}}$ ({by replacing $Q_{i}$ by $a_{i}\, \dfrac{Q_{i}}{a_{i}}, i=1,2$}). By the relation $a_{2}=\dfrac{a_{1}}{\alpha}$, it is also an analytic function of $a_{1}, \alpha, \dfrac{Q_{1}}{a_{1}}, \dfrac{Q_{2}}{a_{2}}$. The calculation of $F_{sec}^{n}$ from the power series of $\mathcal{F}_{pert}$ in $\alpha$ naturally leads to the expansion
$$\mathcal{F}_{sec}^{n}=\sum^{\infty}_{i=2}\, \mathcal{F}_{sec}^{n, i}\, \alpha^{i+1}.$$
 By construction, 
 $$\mathcal{F}_{sec}^{n}-\mathcal{F}_{sec}^{1}=\langle \mathcal{F}_{comp,1} \rangle+ \cdots +\langle\mathcal{F}_{comp,n-1}\rangle+\mathcal{F}_{comp,n}=O(\alpha^{\frac{9}{2}}),$$
which implies
$$\mathcal{F}_{sec}^{n, 2}=\mathcal{F}_{sec}^{1, 2},\,\, \mathcal{F}_{sec}^{n, 3}=\mathcal{F}_{sec}^{1, 3},\quad \forall n \in \Z_{+}.$$

The following lemma shows {in particular} that $\mathcal{F}_{quad}=\mathcal{F}_{sec}^{1,2}$ has an additional circle symmetry, and is thus integrable. 
\begin{lem}\label{lem: quadrupolar regularized system}
The function $\mathcal{F}_{sec}^{1,2}$ depends non-trivally on $G_{2}$, but is independent of $g_{2}$. The function $\mathcal{F}_{sec}^{1,3}$ depends non-trivially on $g_{2}$. 
\end{lem}
\begin{proof}
Consider the system $F_{pert}$, which is naturally a function of $l_{1}$ (resp. $u_{1}$), the mean anomaly (resp. eccentric anomaly) of the inner Keplerian ellipse, when $l_{1}$ (resp. $u_{1}$) is well defined. We define
$$F_{sec}^{1}=\dfrac{1}{4 \pi^{2}}\int_{0}^{2 \pi} \int_{0}^{2 \pi} F_{pert} \,d l_{1} d l_{2}.$$
and develop $F_{sec}^{1}$ in powers of $\alpha=\dfrac{a_{1}}{a_{2}}$:
$$F_{sec}^{1}=\sum^{\infty}_{i=2}\, F_{sec}^{1, i}\, \alpha^{i+1}.$$
We see in \cite{LaskarBoue} that $F_{sec}^{1,2}$ depends non-trivially on $G_{2}$ (through $e_{2}$), but is independent of $g_{2}$, and $F_{sec}^{1,3}$ depends non-trivially on $g_{2}$. 

To conclude, it suffices to notice that, aside from degenerate inner ellipses, we have 
$$\mathcal{F}_{sec}^{1,2}=K.S.^{*} \left(a_{1} \cdot F_{sec}^{1,2}\right), \,\, \mathcal{F}_{sec}^{1,3}=K.S.^{*} \left(a_{1} \cdot F_{sec}^{1,3}\right),$$
which are deduced from
\begin{align*}
K.S.^{*}\left(a_{1}\cdot F_{sec}^{1}\right) &= K.S.^{*}\left(\dfrac{1}{4 \pi^{2}} \int_{0}^{2 \pi} \int_{0}^{2 \pi} \|Q_{1}\|\, F_{pert} \,d u_{1} d l_{2}\right)\\
                                                           &= K.S.^{*}\left(\dfrac{1}{4 \pi^{2}} \int_{0}^{2 \pi} \int_{0}^{2 \pi} \|Q_{1}\|\, F_{pert} \,d \vartheta_{0} d l_{2}'\right)   \\
                                                           &= \dfrac{1}{4 \pi^{2}} \int_{0}^{2 \pi} \int_{0}^{2 \pi} K.S.^{*}\left(\|Q_{1}\|\, F_{pert}\right) \,d \vartheta_{0} d l_{2}'\\ 
                                                           &=\mathcal{F}_{sec}^{1}.                    
\end{align*}
In the above, we have used the following facts:
\begin{itemize}
\item $a_{1} \, d l_{1} = \|Q_{1}\|\, d u_{1}$; 
\item $\vartheta_{0}, l_{2}'$ differs from $u_{1}, l_{2}$ only by some phase shifts depending on neither of these angles.
\end{itemize}
\end{proof}

Better integrable approximating systems are obtained by eliminating the dependence of $g_{2}$ in $\mathcal{F}_{sec}^{n}$. Let $\nu_{quad, 2}=\dfrac{\partial \mathcal{F}_{quad}}{\partial g_{2}}$ be the frequency of $g_{2}$ in the system $\mathcal{F}_{quad}$. As a non-constant analytic function, $\nu_{quad, 2}$ is non-zero almost everywhere in $\tilde{\mathcal{P}}^{n}$, and the set $\check{\mathcal{P}}_{\varepsilon_{0}}^{n} \subset \tilde{\mathcal{P}}^{n}$ characterized by the condition $|\nu_{quad, 2}| > \varepsilon_{0}$ has relative measure tending to 1 in $\tilde{\mathcal{P}}^{n}$ when $\varepsilon_{0} \to 0$. We shall show in Subsection \ref{Subsection: Quadrupolar system and its dynamics} that, for small $\varepsilon$, the set $\check{\mathcal{P}}_{\varepsilon_{0}}^{n}$ contains the region of the phase space that we are interested in. After fixing $\varepsilon_{0}$, there exists an open subset $\hat{\mathcal{P}}_{\varepsilon_{0}}^{n} \subset \check{\mathcal{P}}_{\varepsilon_{0}}^{n}$ whose relative measure in $\check{\mathcal{P}}_{\varepsilon_{0}}^{n}$ tends to 1 when $\alpha \to 0$, and a symplectomorphism $\psi^{n'}: \hat{\mathcal{P}}_{\varepsilon_{0}}^{n} \to \psi^{n'}(\hat{\mathcal{P}}_{\varepsilon_{0}}^{n})$ which is $|\alpha|$-close to the identity, such that  
$$\psi^{n' *} \phi^{n *} \mathcal{F}=\mathcal{F}_{Kep}+\overline{\mathcal{F}_{sec}^{n, n'}}+\mathcal{F}_{pert}^{n,n'},$$
with $\overline{\mathcal{F}_{sec}^{n, n'}}=\alpha^{3} F_{quad} + O(\alpha^{4})$ invariant under the $SO(3)$-symmetry {and} independent of $\vartheta_{0}, l_{2}', g_{2}$ (thus integrable), and $\mathcal{F}_{pert}^{n,n'}=O(\alpha^{\min\{n'+1, \frac{3(n+2)}{2}\}})$. For any $(n, n'-2) \in \Z_{+}^{2}$, the function $\mathcal{F}_{Kep}+\overline{\mathcal{F}_{sec}^{n, n'}}$ is {always} (conjugate to) an integrable approximating system of $\mathcal{F}$.

The symplectomorphism $\psi^{n'}$ is constructed by successive steps of elimination of $g_{2}$ analogous to the proof of Prop \ref{prop: fast angle averaging}, and is dominated by $\psi^{3}$ when $\alpha$ is small enough. We shall describe the choice of $\psi^{3}$ more precisely when needed (Section \ref{Section: Transversality}). 

\section{Dynamics of the Integrable Approximating System} \label{Section: Dynamics of the Integrable Normal Forms}
For small enough $\alpha$ and large enough $n, n'$, the system $\psi^{n' *} \phi^{n *} \mathcal{F}$ (to which $\mathcal{F}$ is conjugate) is a small perturbation of the integrable approximating system $\mathcal{F}_{Kep}+\overline{\mathcal{F}_{sec}^{n, n'}}$ in which the fast motion is dominated by $\mathcal{F}_{Kep}$, while secular evolution of the (physical regularized) ellipses is governed by $\overline{\mathcal{F}_{sec}^{n, n'}}$. 

\subsection{Local Reduction procedure}
To prove Theorem \ref{Theo: Main}, we shall be only interested in those invariant tori of $\mathcal{F}_{Kep}+\overline{\mathcal{F}_{sec}^{n, n'}}$ close to double inner collisions, their geometry and their torsion. We first reduce the system by its known continuous symmetries.

After fixing $\mathcal{P}_{0}, L_{2} > 0$ (\emph{i.e.} fixing $a_{1}, a_{2}>0$) and being reduced by the Keplerian $\T^{2}$-action, the functions $\mathcal{F}_{Kep}+\overline{\mathcal{F}_{sec}^{n, n'}}$ are naturally defined on a subset of the direct product of the space of (inner) centered ellipses with fixed semi major axis with the space of (outer) Keplerian ellipses (\emph{i.e.} bounded orbits of the Kepler problem, which are possibly degenerate or circular ellipses with one focus at origin) with fixed semi major axis. The constant term $\mathcal{F}_{Kep}$ plays no role in the reduced dynamics and is omitted from now on. The circle fibres of $K. S.$ are further reduced out by considering the functions $\overline{\mathcal{F}_{sec}^{n, n'}}$ as defined on subsets of the \emph{secular space}, \emph{i. e.} the space of pairs of (possibly degenerate) Keplerian ellipses with fixed semi major axes. A construction originated by W. Pauli \cite{Pauli} (See also \cite[Lecture 4]{Albouy}) shows that, when no further restrictions are imposed on the two ellipses, the secular space is homeomorphic to $S^{2} \times S^{2} \times S^{2} \times S^{2}$. 

By fixing the direction of the total angular momentum $\vec{C}$ to be vertical (which implies in particular that the two node lines of the two Keplerian ellipses coincide), we restrict the $SO(3) \times SO(2)$-symmetry of $\overline{\mathcal{F}_{sec}^{n, n'}}$ to a (Hamiltonian) $\T^{2}$-symmetry with moment mapping $(C:=\|\vec{C}\|, G_{2})$. Fixing $C$ and $G_{2}$ and then reducing by the $\T^{2}$-symmetry accomplishes the reduction procedure. We assume that $C$ and $G_{2}$ are fixed properly so that the reduced space is 2-dimensional. 

By the triangle inequality, the norm of the angular momentum of the inner Keplerian ellipse $G_{1}$ satisfies 
$$G_{1} \ge |C-G_{2}| \triangleq G_{1, min}.$$
When $C \neq G_{2}$, this inequality bounds $e_{1}$ away from $1$. The inequality becomes equality exactly when $G_{1}=C-G_{2}$ or $G_{1}=G_{2}-C$, corresponding respectively to direct and retrograde coplanar motions. 

A local analysis of the reduced space near coplanar motions suffices for our purpose (c.f. Figure \ref{ContinuityNearDegenerate}). {The reduction procedure of the (free) $SO(2) \times SO(2)$-symmetry for non-coplanar pairs of ellipses is just a combination of Jacobi's node reduction together with the identification of all the outer ellipses with the same angular momentum but different pericentre directions; the coplanar pairs with fixed inner and outer angular momenta are reduced to a point by identifying all the pericentre directions of the inner and outer ellipses. We thus obtain the following: }

When $C \neq G_{2}$, locally near the set $\{G_{1}=G_{1, min}\}$, the reduced space is a disc containing the point corresponding to $G_{1}=G_{1, min}$. The rest of the disc is foliated by the closed level curves of $G_{1}$ (for $G_{1} > G_{1, min}$).

When $C=G_{2}$, for small $G_{1}$, the two ellipses are coplanar only if the inner ellipse degenerates (to a line segment), corresponding to a point after reduction. The reduced space is a disc containing this point; it also contains a line segment corresponding to degenerate inner Keplerian ellipses slightly inclined with respect to the outer ellipse. 

\subsection{Coordinates on the reduced spaces}
To analyze the reduced dynamics of $\overline{\mathcal{F}_{sec}^{n, n'}}$, we need to find appropriate coordinates in the reduced space. For this purpose, to start with the regular coordinates for the inner motion is not convenient, {because they do not naturally descend to Darboux coordinates in the quotient}. Instead, we use Delaunay coordinates $(L_{1}, l_{1}, G_{1}, g_{1}, H_{1}, h_{1})$ for the inner (physical) Keplerian ellipse (with modified masses), which may equally be seen as Darboux coordinates on an open subset of $V^{0}$ where all these elements are well-defined for the inner Keplerian ellipse. 

We observe that fixing $\mathcal{P}_{0}$ (defined in Subsection \ref{Subsection: Regular Coordinates}) and $L_{2}$ is equivalent to fixing $L_{1}(\mathcal{P}_{0}, L_{2})$ and $L_{2}$, and defines a 10-dimensional submanifold of $V^{0} \times T^{*} \R^{3}$, on which the symplectic form 
$$d L_{1} \wedge d l_{1} + d G_{1} \wedge d g_{1} + d H_{1} \wedge d h_{1} + d L_{2} \wedge d l_{2} + d G_{2} \wedge d g_{2} + d H_{2} \wedge d h_{2}$$
restricts to 
$$d G_{1} \wedge d g_{1} + d H_{1} \wedge d h_{1} + d G_{2} \wedge d g_{2} + d H_{2} \wedge d h_{2}$$
with, {thanks to the modification of the masses}, the latter's kernel containing exactly the vectors tangent to the (regularized) {inner} orbits at each point (c.f. Lemma \ref{lem: physical ellipse}), and thus descends to the quotient space by the Keplerian $\T^{2}$-action. We thus obtain a set of Darboux coordinates in a dense open subset of the secular space. 

To reduce out the $SO(3)$-symmetry, we use Jacobi's elimination of the nodes: we fix $\vec{C}$ vertical\footnote{This choice of direction of $\vec{C}$ is convenient, but not essential: the reduced dynamics is the same regardless of the direction of $\vec{C}$.} (which implies that $h_{1}=h_{2}+\pi$ and $H_{1}+H_{2}=C$) and reduce by the conjugate $SO(2)$-symmetry to get a set of Darboux coordinates $(G_{1}, g_{1}, G_{2}, g_{2})$ in the quotient space. Due to the lack of the node lines, the angles $g_{1}, g_{2}$ are not well-defined when the inner ellipse degenerates. Nevertheless, these coordinates are sufficient for what follows.
The $SO(2)$-symmetry of rotating the outer ellipse in its orbital plane can be symplectically reduced by fixing $G_{2}$ {in addition}. The pair $(G_{1}, g_{1})$ is a set of Darboux coordinates in an open subset of the 2-dimensional quotient space.

\subsection{The quadrupolar system and its dynamics}\label{Subsection: Quadrupolar system and its dynamics}

The system $\overline{\mathcal{F}_{sec}^{n, n'}}$ is an $O(\alpha^{4})$-perturbation of $\alpha^{3} \mathcal{F}_{quad}$. Let us first analyze the quadrupolar dynamics, \emph{i.e.} the dynamics of $\mathcal{F}_{quad}$.  Let $\mu_{quad}=\dfrac{m_{0} m_{1} m_{2}}{m_{0}+m_{1}}$. In coordinates $(G_{1}, g_{1})$ with parameters $L_{1}, L_{2}, C, G_{2}$, the function $$\mathcal{F}_{quad}(G_{1}, g_{1}; L_{1}, L_{2}, C, G_{2})$$
is equal to
\small
\begin{align*}
&-\dfrac{\mu_{quad} L_{2}^{3}}{8 G_2^3} \left\{3\dfrac{G_1^2}{L_1^2} \Bigl[1+\dfrac{(C^2-G_1^2-G_2^2)^2}{4 G_1^2 G_2^2}\Bigr]+ 15 \Bigl(1-\dfrac{G_1^2}{L_1^2}\Bigr)\Bigl[\cos^2{g_1}+\sin^2{g_1} \dfrac{(C^2-G_1^2-G_2^2)^2}{4 G_1^2 G_2^2}\Bigr] -6\Bigl(1-\dfrac{G_1^2}{L_1^2}\Bigr)-4\right\},
\end{align*}
\normalsize
which differs from $F_{quad}$ by a non-essential factor $a_{1}$. We have separated the variables from the parameters of a system by a semicolon. The dynamics of $F_{quad}$ has been extensively studied by Lidov and Ziglin in \cite{LidovZiglin}, from which the dynamics of $\mathcal{F}_{quad}$ is deduced directly. 

\begin{rem}The relation between $\mathcal{F}_{quad}$ and $F_{quad}$ (see the proof of Lemma \ref{lem: quadrupolar regularized system}) also justifies the fact that $F_{quad}$ can be extended analytically through degenerate inner ellipses.
\end{rem}

For $|C-G_{2}|$ positive but small, locally the reduced secular space is foliated by closed curves around the point $\{G_{1}=G_{1, min}\}$ corresponding to coplanar motions. This is deduced from \cite{LidovZiglin} by noticing that $(G_{1}, g_{1})$ are regular coordinates outside the point $\{G_{1}=G_{1, min}\}$. (c.f. Figure \ref{ContinuityNearDegenerate})

When $C=G_{2}$, the Hamiltonian $\mathcal{F}_{quad}$ takes the form
\small
\begin{align*}
 \mathcal{F}_{quad}=-\dfrac{\mu_{quad} L_{2}^{3}}{8 G_2^3} \left\{3\dfrac{G_1^2}{L_1^2} \Bigl[1+\dfrac{G_1^2}{4  G_2^2}\Bigr] + 15 \Bigl(1-\dfrac{G_1^2}{L_1^2}\Bigr)\Bigl[\cos^2{g_1}+\sin^2{g_1} \dfrac{G_1^2}{4 G_2^2}\Bigr] -6\Bigl(1-\dfrac{G_1^2}{L_1^2}\Bigr)-4\right\}
\end{align*}
\normalsize
which admits the symmetry $(G_{1}, g_{1}) \sim (-G_{1}, \pi-g_{1})$ and is a well-defined analytic function on the cylinder 
$$\mathcal{D}:=\{(G_{1}, g_{1}) \in \R \times \R/2 \pi \Z: -\min\{L_{1}, C+G_{2}\} < G_{1} < \min\{L_{1}, C+G_{2}\} \},$$
which is a (branched) double cover of a neighborhood of the line segment $\{G_{1}=0\}$ in the reduced space. Moreover, the 2-form $d G_{1} \wedge d g_{1}$ extends uniquely to a (non-degenerate) 2-form invariant under the symmetry $(G_{1}, g_{1}) \sim (-G_{1}, \pi-g_{1})$ on $\mathcal{D}$, and makes $\mathcal{D}$ into a symplectic manifold. The flow of the Hamiltonian function $\mathcal{F}_{quad}(G_{1}, g_{1}; C=G_{2}, L_{1}, L_{2})$ in $\mathcal{D}$ is thus interpreted as the lift of the quadrupolar flow in the reduced space. Therefore, rather than choosing coordinates in the reduced space and studying the quadrupolar flow directly, we shall study the dynamics of $\mathcal{F}_{quad}(G_{1}, g_{1}; C=G_{2}, L_{1}, L_{2})$ in $\mathcal{D}$ on which we have global Darboux coordinates $(G_{1}, g_{1})$ (c.f. Figure \ref{ReducedCriticalQuadrupolarSpace}). 

Let $I$ be the mutual inclination of two Keplerian orbits. The condition $C=G_{2}$ implies $\cos I =-\dfrac{G_{1}}{2 C}$. In particular, when $G_{1}=0$, the limiting orbital plane of the inner Keplerian ellipse is perpendicular to the outer orbital plane. The coplanar case is thus characterized by $(G_{1}=0, g_{1}=0) \sim (G_{1}=0, g_{1}=\pi)$, which are two elliptic equilibria for the lifted flow in $\mathcal{D}$ surrounded by periodic orbits. These periodic orbits meet the line $\{G_{1}=0\}$ transversely with an angle independent of $\alpha$. Being reduced by the discrete symmetry $(G_{1}, g_{1}) \sim (-G_{1}, \pi-g_{1})$, the two elliptic equilibria in $\mathcal{D}$ descend to an elliptic equilibrium $E$ surrounded by periodic orbits in the reduced space, and these periodic orbits meet the set $\{G_{1}=0\}$ transversely. {The $\Z_{2}$-action $(G_{1}, g_{1}) \sim (-G_{1}, \pi-g_{1})$ is free everywhere except for the two points $(G_{0}=0, g_{1}=\pm \dfrac{\pi}{2})$. These two points descend to two singular points in the quotient space. } (c.f. Figure \ref{ContinuityNearDegenerate})

\begin{figure}
\centering
\includegraphics[width=2.5in]{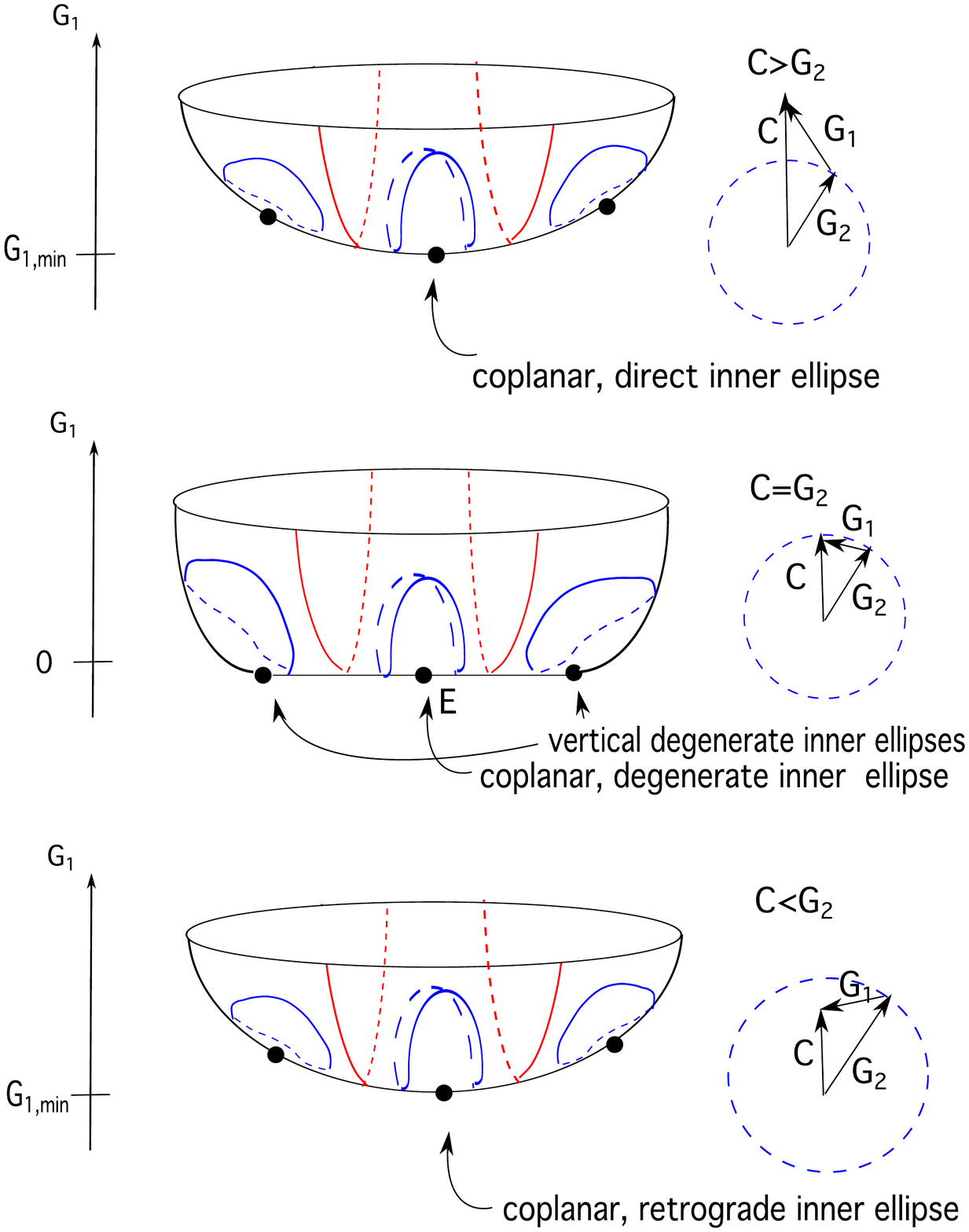} 
\caption{The reduced quadrupolar flow near $G_{1}=G_{1,max}$. We are interested only in the regions near the depicted coplanar equilibria inside the separatrices. } \label{ContinuityNearDegenerate}
\end{figure}

\begin{figure}
\centering
\includegraphics[width=4.5in]{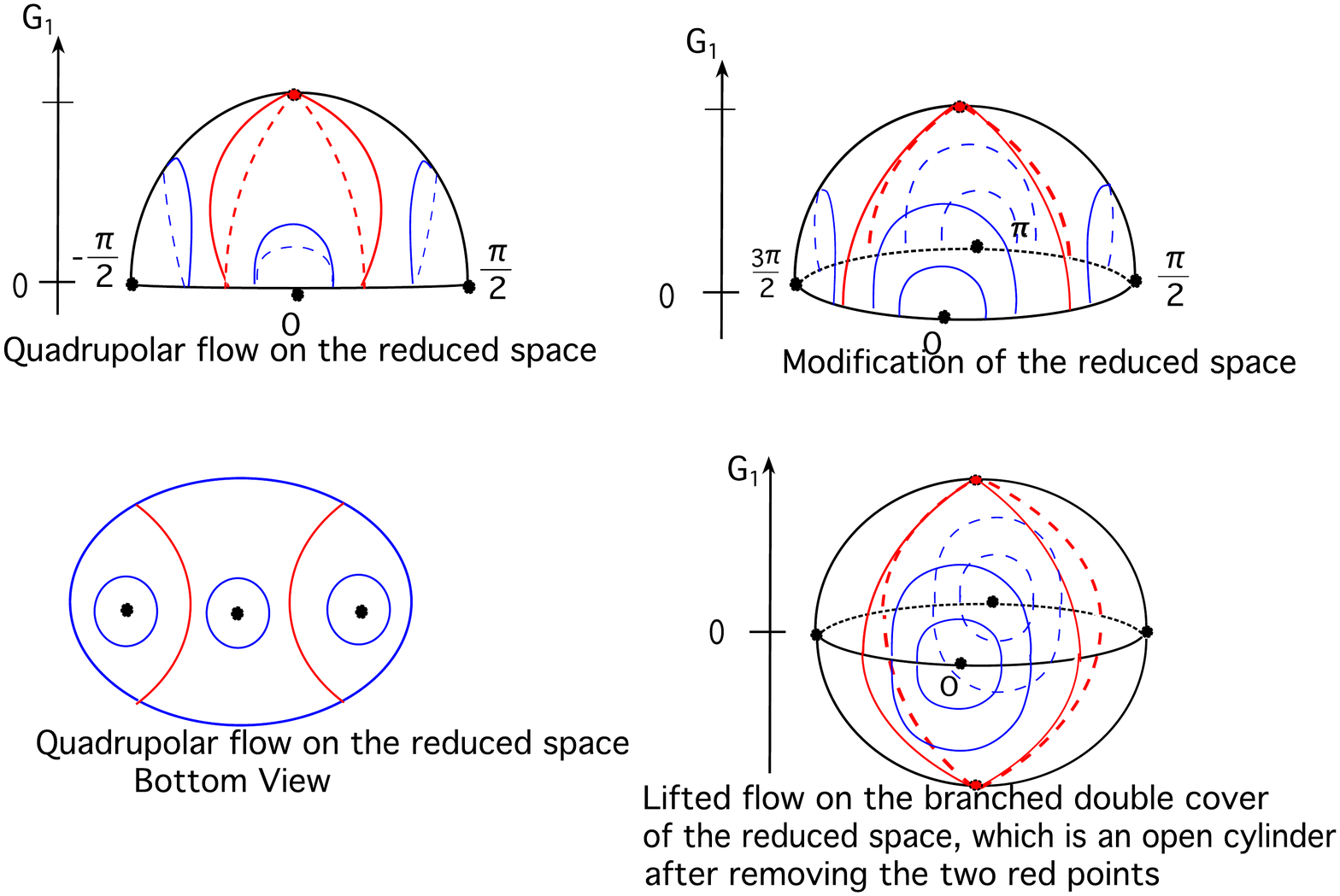} 
\caption{The flow foliation on the reduced space for $C=G_{2}$ and its branched double cover. {Note that after reduction, $G_{1}=\min\{L_{1}, C+G_{2}\}$ also descends to a point.}} \label{ReducedCriticalQuadrupolarSpace}
\end{figure}

In Subsection \ref{Subsection: Elimination of g_{2}}, we have defined the set $\check{\mathcal{P}}_{\varepsilon_{0}}^{n} \subset \tilde{\mathcal{P}}^{n}$ by the condition $|\nu_{quad, 2}|=\left|\dfrac{\partial \mathcal{F}_{quad}}{\partial G_{2}}\right| > \varepsilon_{0}$. The function $\dfrac{\partial \mathcal{F}_{quad}}{\partial G_{2}}\Bigl|_{C=G_{2}}$ being regarded as a function on $\mathcal{D}$, we find by setting $G_{1}=0$ that
$$\dfrac{\partial \mathcal{F}_{quad}}{\partial G_{2}}\Bigl|_{C=G_{2}, G_{1}=0}=-\dfrac{15 \mu_{quad} L_{2}^{3}}{8 G_{2}^{4}} (3-4 \cos^{2} g_{1}).$$ 
This shows (by passing to the quotient of the symmetry $(G_{1}, g_{1}) \sim (-G_{1}, \pi-g_{1})$) that for $\varepsilon_{0}$ small enough, after being reduced by the $\T^{3} \times SO(3)$-symmetry of the quadrupolar system, $\check{\mathcal{P}}_{\varepsilon_{0}}^{n}$ contains a neighborhood of $E$ (whose size is independent of $\alpha$). 

The following lemma enable us to deduce the local dynamics of $\mathcal{F}_{sec}^{n, n'}$ from that of $\mathcal{F}_{quad}$.

\begin{lem} The equilibrium $E$ is of Morse type. 
\end{lem}
\begin{proof} It is enough to investigate the equilibrium $(G_{1}=0, g_{1}=0)$ of the lifted flow in $\mathcal{D}$, at which the Hessian of $\mathcal{F}_{quad}(G_{1}, g_{1}; C=G_{2}, L_{1}, L_{2})$ equals to $\dfrac{45}{8} \dfrac{ \mu_{quad}^{2} L_{2}^{6}}{G_{2}^{6} L_{1}^{2}}$.
\end{proof}

By continuity, the coplanar equilibria $\{G_{1}=G_{1, min}\}$ are of Morse type for small $G_{1, min}=|C-G_{2}|$. Consequently, $\overline{\mathcal{F}_{sec}^{n, n'}}$ is orbitally conjugate to $\mathcal{F}_{quad}$ in neighborhoods of these coplanar equilibria for small enough $\alpha$.

\section{Application of a KAM theorem} \label{Section: Application of KAM theorem}
\subsection{Iso-energetic KAM theorem}
For $p \ge 1$, consider the phase space $\R^{p} \times
\T^p=\{(I,\theta)\}$ endowed with the
standard symplectic form $d I \wedge d \theta$. All
mappings are assumed to be analytic except when explicitly mentioned
otherwise.

Let $\delta > 0$, $\varpi \in \R^p$. Let $B_{\delta}^{p}$ be the $p$-dimensional closed ball with radius $\delta$ centered at the origin in $\R^{p}$, and $N_{\varpi}= N_{\varpi}
(\delta)$ be the space of Hamiltonians $N \in C^{\omega}(\T^p
\times B_{\delta}^{p},\R)$ of the form
$$N= c + \langle \varpi,I \rangle+\langle A(\theta), I \otimes I \rangle + O(|I|^{3}),
$$ 
with $c\in \R$ and $A \in C^{\omega}(\T^p, \R^{p} \otimes \R^p)$; the
Lagrangian torus $\T^p \times \{0\}$ is an invariant
Lagrangian $\varpi$-quasi-periodic torus of $N$ with energy $c$.

Let $\bar{\gamma} > 0$ and $\bar{\tau} > p-1$, and let $|\cdot|_{2}$ be the $\ell^2$-norm on
$\Z^{p}$. Let
$HD_{\bar\gamma,\bar\tau}$ be the set of vectors $\varpi$ satisfying the
following homogeneous Diophantine conditions:
$$\forall k\in \Z^p \setminus \{0\}, \quad |k \cdot \varpi | \geq \bar\gamma\, |k|_{2}^{-\bar\tau}.$$ 
 {Let $\|\cdot\|_{s}$ be the $s$-analytic norm of an analytic function, \emph{i.e.}, the supremum norm of its analytic extension to 
$$\{x \in \C^{p} \times \C^{p}/\Z^{p}: \exists \, y \in \R^{p} \times \T^{p}, \hbox{ such that } |x-y| \le s \}$$ 
of its (real) domain in the complexified space $\C^{p} \times \C^{p}/\Z^{p}$ with ``radius'' $s$.} 

\begin{theo} \label{Theo: KAM} Let $\varpi^o \in
  HD_{\bar\gamma,\bar\tau}$ and $N^o \in N_{\varpi^o}$. For
 some $d>0$ small enough,  there exists $\varepsilon>0$ such that for every
  Hamiltonian $N \in C^\omega(\T^p \times B_\delta^{p})$ such that
  $$\|N-N^o\|_d \leq \varepsilon,$$
 there exists a vector $(\varpi, c)$ satisfying the following
  properties:
  \begin{itemize}
  \item the map {$N \mapsto (\varpi, c) $ is of class $C^\infty$ and
    is $\varepsilon$-close to $(\varpi^o, c^{o})=(\varpi(N^{o}), c(N^{o}))$} in the
    $C^\infty$-topology;
  \item if $\varpi (N) \in HD_{\bar\gamma,\bar\tau}$, $N$ is 
    symplectically analytically conjugate to a Hamiltonian $${c(N) + \langle \varpi(N),I \rangle+\langle A(\theta)(N), I \otimes I \rangle + O(|I|^{3})\in
    N_{\varpi}.}$$ 
  \end{itemize}
 Moreover, $\varepsilon$ can be chosen of the form $\hbox{Cst} \, \bar\gamma^{k}$
  (for some $\hbox{Cst}>0$, $k\geq 1$) when $\bar\gamma$ is small.
\end{theo}

This theorem is an analytic version of the $C^{\infty}$ ``hypothetical conjugacy
theorem'' of~\cite{FejozStability} (for Lagrangian tori). We refer to \cite{FejozMoser} for its complete proof. 

We now consider families of Hamiltonians $N^o_\iota$ and $N_\iota$ depending  analytically (actually $C^1$-smoothly would suffice) on some parameter
$\iota \in B_1^{p}$. Recall that for each $\iota$, $N^o_\iota$ is
of the form
$$N^o_\iota= c^o_\iota + \langle \varpi^o_\iota,I \rangle +\langle A_{\iota}(\theta), I \otimes I \rangle
+ O(|I|^{3}).$$

With the aim of finding zero-energy invariant tori of $\mathcal{F}$ (recall that it is only on $\mathcal{F}=0$ that the dynamics of $\mathcal{F}$ extends that of $F$), we now deduce an iso-energetic KAM theorem from Theorem \ref{Theo: KAM}. Denote by $[\cdot]$ the projective class of a vector.  Let 
$$D^o = \left\{(c^o_\iota,[\varpi^o_\iota]):\, c^o_\iota=c^{o}_{0}=c^{o}, \,
  \varpi_\iota^o \in HD_{2\bar\gamma,\bar\tau}, \,
  \iota \in B_{1/2}^{p}\right\};$$ 
note that the factor $2$ in the
Diophantine constant $2\bar\gamma$ is meant to take care of the fact
that along a given projective class, locally the constant $\bar\gamma$
may worsen a little (we will apply Theorem~\ref{Theo: KAM} with Diophantine
constants $(\bar\gamma,\bar\tau)$). {By Theorem~\ref{Theo: KAM}, the mapping $
\iota \mapsto (\varpi_{\iota}, c_{\iota})=(\varpi(N_{\iota}), c(N_{\iota}))$ is $C^{\infty}$ and is $\varepsilon$-close to $(\varpi^{o}_{\iota}, C^{o}_{\iota})$.}

\begin{cor}[Iso-energetic KAM theorem] \label{cor: isoe kam}
  Assume that the map 
  $$B_1^{p} \rightarrow \R \times \mathbf{P}(\R^{p}), \quad \iota \mapsto
  (c^{o}_\iota,[\varpi^o_\iota])$$ is a diffeomorphism onto
  its image.   If $\varepsilon$ is small enough and if for some $d >0$, we have $\|N_\iota-N^o_\iota\|_{d} <
  \varepsilon$ for each $\iota$, the following holds:

For every $(c^{o},\nu^{o}) \in D^o$, there exists a unique $\iota \in
  B^{p}_1$ such that $(c_{\iota}, [\varpi_{\iota}])=(c^{o},\nu^{o})$, and $N_\iota$ is symplectically (analytically) conjugate to some
  $N'_\iota \in N_{\varpi_\iota,\beta_\iota}$ of the form
  $$N'_\iota= c^{o} + \langle \varpi_\iota,I \rangle+\langle A_{\iota} (\theta), I \otimes I \rangle + O(|I|^{3}).$$ 
Moreover, there exists $\bar{\gamma}>0, \bar{\tau}>p-1$, such that the set
  $$\{\iota \in B_{1/2}^{p}: \, c_{\iota}=c^{o},\, \varpi_\iota \in
  HD^o\}$$ has positive $(p-1)$-dimensional Lebesgue measure.
\end{cor}  
\begin{proof} 
From the hypothesis, the image of the restriction to $\{\iota: \, c^{o}_{\iota}=c^{o}\}$ of the mapping $\iota \mapsto \varpi^{o}_{\iota}$ is a $(p-1)$-dimensional smooth manifold, diffeomorphic to a subset of $\mathbf{P}(\R^{p})$ with non-empty interior, hence it contains a positive measure set of Diophantine vectors. Therefore there exists $\bar{\gamma}>0, \bar{\tau} >  p-1$, such that the set $D^o$ has positive {$(p-1)$}-dimensional measure. 

Moreover, $D^{o} \subset D' =\left\{(c^{o},[\varpi_\iota]): \,
  \varpi_\iota \in HD_{\bar\gamma,\bar\tau}, \,
  \iota \in B_{2/3}^{p}\right\}$. Indeed, if $(c^{o}, [\varpi^{o}_{\iota^{o}}]) \in D^{o}, \iota^{o} \in B_{1/2}$, then there exists $\iota' \in B_{2/3}$ such that $(c^{o}, [\varpi^{o}_{\iota^{o}}])=(c^{o}, [\varpi_{\iota'}])$. If $\varepsilon$ is small enough, $\varpi_{\iota'}$ is close enough to $\varpi^{o}_{\iota^{o}} \in HD_{2\bar{\gamma}, \bar{\tau}}$, hence belongs to $HD_{\bar{\gamma}, \bar{\tau}}$, and $(c^{o}, [\varpi_{\iota'}]) \in D'$.
If $\varepsilon$ is small, the mapping $\iota \mapsto (c_\iota,[\varpi_\iota])$ is $C^{1}$-close to $\iota \mapsto (c^{o}_\iota,[\varpi^o_\iota])$, hence it is a diffeomorphism, and the image of its restriction to $B_{2/3}^{p}$ contains the set $D'$.
  
The first assertion then follows from Theorem~\ref{Theo: KAM}. Since the map $\iota \mapsto   (c_{\iota},[\varpi_\iota])$ is smooth, the pre-image of a set of positive $(p-1)$-Lebesgue measure has positive $(p-1)$-dimensional Lebesgue measure. 
\end{proof}

\begin{condition} \label{isoenergetically non-degenerate}
When an integrable Hamiltonian $K^o=K^{o}(I)$ depends only on the action variables $I$, we may set $N^o_\iota (I) := K^o(\iota + I)$. The iso-energetic non-degeneracy of $N^{o}_{\iota}$ is just the non-degeneracy of the bordered Hessian
$$
\mathscr{H}^{B}(K^{o})(I)=
\begin{bmatrix}
 0 & {K^{o}}'_{I_{1}} & \cdots & {K^{o}}'_{I_{p}} \\
  {K^{o}}'_{I_{1}}& {K^{o}}''_{I_{1}, I_{1}} & \cdots & {K^{o}}''_{I_{1}, I_{p}} \\
  \vdots  & \vdots  & \ddots & \vdots  \\
  {K^{o}}'_{I_{p}} & {K^{o}}''_{I_{p}, I_{1}} & \cdots & {K^{o}}''_{I_{p}, I_{p}}
\end{bmatrix}$$
(in which ${K^{o}}'_{I_{i}}=\dfrac{\partial K^{o}}{\partial I_{i}}, {K^{o}}''_{I_{i}, I_{j}}=\dfrac{\partial^{2} K^{o}}{\partial I_{i} \partial I_{j}}$), i.e.
$$|\mathscr{H}^{B}(K^{o})(I)| \neq 0.$$
When this is satisfied, Corollary \ref{cor: isoe kam} asserts the persistence under sufficiently small perturbation of a set of Lagrangian invariant tori {(with fixed energy $c_{0}$)} of $N^o=K^{o}(I)$ parametrized by a positive $(p-1)$-Lebesgue measure set in the action space. These invariant tori form a set of positive measure in the energy surface of the perturbed system with energy $c_{0}$. 

Moreover, if the system $K^{o}(I)$ is properly degenerate, say for $I=(I^{(1)}, I^{(2)},\cdots, I^{(N)})$, $0< d_{1} <  d_{2}, \cdots, < d_{N} $, we have 
 $$K^{o}(I)=K^{o}_{1}(I^{(1)})+\epsilon^{d_{1}} K^{o}_{2}(I^{(1)}, I^{(2)}) + \cdots +\epsilon^{d_{N}} K^{o}_{N}(I),$$

{then, by replacing entries of the matrix $\mathscr{H}^{B}(K^{o})(I)$ by their orders in $\epsilon$, we obtain
$$\begin{bmatrix}
 0 & 1 & \epsilon^{d_{1}} & \cdots & \epsilon^{d_{N}} \\
  1& 1 & \epsilon^{d_{1}} & \cdots & \epsilon^{d_{N}} \\
  \epsilon^{d_{1}} & \epsilon^{d_{1}} & \epsilon^{d_{1}} & \cdots & \epsilon^{d_{N}} \\
  \vdots & \vdots & \vdots  & \ddots & \vdots  \\
  \epsilon^{d_{N}} & \epsilon^{d_{N}} & \epsilon^{d_{N}} & \cdots & \epsilon^{d_{N}}
\end{bmatrix}$$}
which in particular implies that
 $$|\mathscr{H}^{B}(K^{o})(I)| \neq 0, \, \, \forall \,0<\epsilon<<1 {\Leftarrow} |\mathscr{H}^{B}(K^{o}_{1})(I^{(1)})| \neq 0, |\mathscr{H}(K^{o}_{i})(I^{(i)})| \neq 0, \forall (i)=(2),\cdots, (N).$$
 
The smallest frequency of $K^{o}(I)$ is of order $\epsilon^{d_{N}}$. 
If $K^{o}(I)$ is iso-energetically non-degenerate, then for any $0<\epsilon<<1$, there exists a set of positive $(p-1)$-Lebesgue measure of the action space, such that
the set of projective classes of their frequencies contain a set of positive measure of projective classes of homogeneous Diophantine vectors in 
$$HD_{\epsilon^{d_{N}} \bar{\gamma},\bar{\tau}}:=\{\varpi \in \R^{p}: \forall k\in \Z^p \setminus \{0\}, \quad |k \cdot \varpi | \geq \epsilon^{d_{N}} \bar\gamma\, |k|_{2}^{-\bar\tau}\}$$
whose measure is uniformly bounded from below for $0<\epsilon<<1$: Actually, since for any vector $\nu' \in \R^{p}$,
$$\epsilon^{d_{N}} \nu' \in HD_{\epsilon^{d_{N}} \bar{\gamma},\bar{\tau}} \Leftrightarrow \nu' \in HD_{\bar{\gamma},\bar{\tau}}, $$
for $\epsilon$ sufficiently small, the measure of projective classes of Diophantine frequencies of $K^{o}(I)$ in $HD_{\epsilon^{d_{N}} \bar{\gamma},\bar{\tau}}$ is at least the measure of projective classes of Diophantine frequencies of 
$$K^{o}_{1}(I^{(1)})+K^{o}_{2}(I^{(1)}, I^{(2)}) + \cdots +K^{o}_{N}(I)$$
in $HD_{\bar{\gamma},\bar{\tau}}$, which is independent of $\epsilon$.

Following Theorem \ref{Theo: KAM}, we may thus set $\varepsilon=\hbox{Cst}\, (\epsilon^{d_{N}}\, \bar{\gamma})^{k}$ for the size of allowed perturbations, for some positive constant $\hbox{Cst}$ and some $k \ge 1$, provided $\bar{\gamma}$ is small. 
\end{condition}

\subsection{Application of the iso-energetic KAM theorem} \label{Subsection: Application of KAM theorem}
After symplectic reduction by the $SO(3)$-symmetry of rotations and the $S^{1}$-action of $K. S.$, let us first show the existence of torsion near $G_{1}=G_{1, min}$, for $G_{1, min}=|C-G_{2}|$ small enough in the system 
$\mathcal{F}_{Kep}+\overline{\mathcal{F}_{sec}^{n, n'}}.$

In view of Condition \ref{isoenergetically non-degenerate}, for small $\alpha$, we verify the iso-energetic non-degeneracy condition by verifying the corresponding non-degeneracy conditions separately for the Keplerian part (with respect to $\mathcal{P}_{0}$ and $L_{2}$) and for the system $\overline{\mathcal{F}_{sec}^{n, n'}}$ reduced further by the Keplerian $\T^{2}$-symmetry. 

\subsubsection{Keplerian part}
The bordered Hessian of 
$$\mathcal{F}_{Kep}=P_{0} \sqrt{\dfrac{2 f_{1} (L_{2})}{\mu_{1}}}-\mu_{1} M_{1}$$
with respect to $\mathcal{P}_{0}$ and $L_{2}$ is non-degenerate.

\subsubsection{Secular non-degeneracy}
Keeping unreduced only the $SO(2)$-symmetry conjugate to $G_{2}$, the periodic orbits in  the corresponding completely reduced 1 degree of freedom system are lifted to invariant 2-tori of $\overline{\mathcal{F}_{sec}^{n, n'}}$ whose frequencies differ from that of the invariant 2-tori of $\alpha^{3} F_{quad}$ only by quantities of order $O(\alpha^{4})$. For small enough $\alpha$, the existence of torsion of these invariant 2-tori of $\overline{\mathcal{F}_{sec}^{n, n'}}$ for any $n, n'$ thus follows from the existence of torsion of invariant 2-tori of $\mathcal{F}_{quad}$. For $|C-G_{2}|$ small enough, we shall verify in Appendix \ref{Appendix: Torsion} the existence of torsion of almost coplanar invariant 2-tori of $\mathcal{F}_{quad}$ close enough to $\{G_{1}=G_{1, min}\}$ (which in particular does not vanish when $C \to G_{2}$).

\subsubsection{Application of the iso-energetic KAM theorem}
We fix $n, n'$ large enough, so that $\mathcal{F}_{pert}^{n,n'}$ is of order $O(\alpha^{4 k+1})$ (the order is chosen so as to fit  Condition \ref{isoenergetically non-degenerate} for sufficiently small $\alpha$). 

The invariant tori of $\mathcal{F}_{Kep}+\overline{\mathcal{F}_{sec}^{n, n'}}$ near $\{G_{1}=G_{1, min}\}$ are smoothly parametrized by $(\mathcal{P}_{0}, L_{2}, \mathcal{J}_{1}, G_{2})$ where $\mathcal{J}_{1}$ designates the region (containing the point $\{G_{1}=G_{1, min}\}$) enclosed by the corresponding periodic orbit of the invariant torus after further reducing by the Keplerian $\T^{2}$-symmetry and the $SO(2)$-symmetry conjugate to $G_{2}$.  For small enough $\alpha$, the above non-degeneracies ensure the existence of a neighborhood $\Omega$ of $\{G_{1}=G_{1, min}\}$ for small enough $G_{1, min}=|C-G_{2}|$, in which the mapping 
$$(\mathcal{P}_{0}, L_{2}, \mathcal{J}_{1}, G_{2}) \mapsto \left(\mathcal{F}_{Kep}+\overline{\mathcal{F}_{sec}^{n, n'}}, [\nu_{n, n'}]\right) $$
is a local diffeomorphism (with energy containing a neighborhood of $0$), where we have denoted by $\nu_{n, n'}$ the frequencies of the invariant 4-tori of $\mathcal{F}_{Kep}+\overline{\mathcal{F}_{sec}^{n, n'}}$. Therefore, there exist $\bar{\gamma} >0, \bar{\tau} \ge 3$, and a set $\Omega'$ of positive measure (whose measure is bounded from below uniformly for small $\alpha$), consisting of $(\alpha^{3} \bar{\gamma}, \bar{\tau})$-Diophantine invariant Lagrangian tori of $\mathcal{F}_{Kep}+\overline{\mathcal{F}_{sec}^{n, n'}}$. For any such torus with parameter $(\mathcal{P}^{o}_{0}, L^{o}_{2}, \mathcal{J}^{o}_{1}, G^{o}_{2})$, there exists $\lambda>0$, such that for $(\overline{\mathcal{P}_{0}}, \overline{L_{2}}, \overline{\mathcal{J}_{1}}, \overline{G_{2}}) \in B_{1}^{4}$, the mapping
$$\Phi_{\lambda}(\overline{\mathcal{P}_{0}}, \overline{L_{2}}, \overline{\mathcal{J}_{1}}, \overline{G_{2}}):=(\mathcal{P}^{o}_{0}+\lambda\, \overline{\mathcal{P}_{0}}, L^{o}_{2}+\lambda\, \overline{L_{2}}, \mathcal{J}^{o}_{1}+\lambda\, \overline{\mathcal{J}_{1}}, G^{o}_{2}+\lambda\, \overline{G_{2}}) \mapsto \left(\mathcal{F}_{Kep}+\overline{\mathcal{F}_{sec}^{n, n'}}, [\nu_{n, n'}]\right) $$
is a diffeomorphism. We suppose in addition that the invariant torus $(\mathcal{P}^{o}_{0}, L^{o}_{2}, \mathcal{J}^{o}_{1}, G^{o}_{2})$ of $\mathcal{F}_{Kep}+\overline{\mathcal{F}_{sec}^{n, n'}}$ has zero energy.

We may now apply Corollary \ref{cor: isoe kam} with 
$$N^{o}=\Phi_{\lambda}^{*}(\mathcal{F}_{Kep}+ \overline{\mathcal{F}_{sec}^{n,n'}}), \, N=\Phi_{\lambda}^{*}(\mathcal{F}_{Kep}+ \overline{\mathcal{F}_{sec}^{n,n'}}+\mathcal{F}_{pert}^{n,n'}), \iota=(\mathcal{P}^{o}_{0}, L^{o}_{2}, \mathcal{J}^{o}_{1}, G^{o}_{2})$$
and enough small $\alpha$. 

In this way, we obtain a set of invariant 4-tori of $\mathcal{F}$ reduced by the $SO(3)$-symmetry, which has positive measure on the energy level $\mathcal{F}=0$. By rotating around $\vec{C}$, these 4-tori give rise to a set of invariant 5-tori of $\mathcal{F}$ (being reduced by the $S^{1}$-fibre symmetry of $K.S.$) with fixed (vertical) direction of $\vec{C}$ in $\Pi_{reg}:=V^{0} \times T^{*}(\R^{3} \setminus \{0\})$. Finally, by rotating $\vec{C}$, we obtain a set of invariant 5- tori of $\mathcal{F}$ in $V^{0} \times T^{*}(\R^{3} \setminus \{0\})$ having positive measure on the energy level $\mathcal{F}=0$. Depending on the commensurability of the frequencies, the flows on these invariant 5-dimensional tori may either be ergodic or be non-ergodic but only ergodic on some invariant 4-dimensional subtori.

\section{Transversality}\label{Section: Transversality}

\subsection{Transversality of the Ergodic tori with the collision set}
The $S^{1}$-fibre action of $K. S.$ is free on the codimension-3 submanifold $\{(0, w, Q_{2}, P_{2}) \in T^{*} \mathbb{H} \setminus \{(0, 0)\} \times T^{*} (\R^{3} \setminus \{0\})\}$ of $\Sigma^{0}$ corresponding to inner double collisions of $F$. The quotient $\mathcal{C}ol$ is thus a codimension-3 submanifold of $\Pi_{reg}:=V^{0} \times T^{*}(\R^{3} \setminus \{0\})$. 

We aim to show that after being reduced by the $S^{1}$-fibre symmetry of $K.S.$, the invariant ergodic tori of $\mathcal{F}$ intersecting $\mathcal{C}ol$ transversely form a set of positive measure in the energy level $\mathcal{F}=0$ in  $\Pi_{reg}$. 

In Subsection \ref{Subsection: Elimination of g_{2}}, we have shown the existence of a symplectic transformation $$\phi^{n} \circ \psi^{n'}: \hat{\mathcal{P}}_{\varepsilon_{0}}^{n} \to \phi^{n} \circ \psi^{n'}(\hat{\mathcal{P}}_{\varepsilon_{0}}^{n}),$$ dominated by $\psi^{3}$ for small $\alpha$, such that 
$$\psi^{n' *} \phi^{n *} \mathcal{F}=\mathcal{F}_{Kep}+\overline{\mathcal{F}_{sec}^{n, n'}}+\mathcal{F}_{pert}^{n,n'}.$$

When $\alpha$ is sufficiently small, let us first show that those invariant 5-tori of $\mathcal{F}_{Kep}+\overline{\mathcal{F}_{sec}^{n, n'}}$ intersecting $\mathcal{C}ol'=(\phi^{3})^{-1} (\mathcal{C}ol)$ transversely form an open set in the energy level $\mathcal{F}_{Kep}+\overline{\mathcal{F}_{sec}^{n, n'}}=0$ in $\Pi_{reg}$.

Denote by $\Pi_{reg}^{\tilde{C}}$ the 11-dimensional submanifold of $\Pi_{reg}$ with $C=\tilde{C}>0$ and by $\underline{\mathcal{C}ol}$ the (transverse) intersection of $\mathcal{C}ol$ and $\Pi_{reg}^{\tilde{C}}$. The intersection of ${\{C=G_{2}\}}$ with $\Pi_{reg}^{\tilde{C}}$ is denoted by $\underline{\{C=G_{2}\}}$. 

{As we are interested in invariant tori in $\mathcal{F}=0$, we could hence fix $\mathcal{F}_{Kep}=0$ (which implies $\mathcal{F}=O(\alpha^{3})$) and then adjust the energy properly}. In the sequel, unless otherwise stated, an invariant torus is always understood as an invariant 5-torus (on which $\vec{C}$ is conserved) of $\mathcal{F}_{Kep}+\overline{\mathcal{F}_{sec}^{n,n'}}$ with $\mathcal{F}_{Kep}=0$. In addition, we suppose that $\vec{C}$ is sufficiently inclined and $\alpha$ is small enough, so that the Delaunay coordinates are well-defined for the outer body. We take any convenient coordinates on $V^{0}$ for the inner body.

\begin{lem} For small enough $\alpha$, any invariant torus in $\underline{\{C=G_{2}\}}$ intersects $\underline{\mathcal{C}ol}$ transversely in $\underline{\{C=G_{2}\}}$.
\end{lem}
\begin{proof} Any such invariant torus is a $O(\alpha)$-deformation of an invariant torus of $\mathcal{F}_{Kep}+\alpha^{3} \,\mathcal{F}_{quad}$ in $\underline{\{C=G_{2}\}}$. After being reduced by the $\T^{2} \times SO(3) \times SO(2)$-symmetry, such an invariant torus of $\mathcal{F}_{Kep}+\alpha^{3} \,\mathcal{F}_{quad}$ descends to a closed orbit intersecting the line segment $\{G_{1}=0\}$ transversely (c.f. Figure \ref{ReducedCriticalQuadrupolarSpace}), therefore it intersects transversely the codimension-1 submanifold of $\underline{\{C=G_{2}\}}$ consisting of degenerate inner ellipses in $\underline{\{C=G_{2}\}}$; moreover, being foliated by the $S^{1}$-orbits of the inner particle of $\mathcal{F}_{Kep}$ ({parametrized by $u_{1}$}), this torus also intersects the codimension-2 submanifold $\underline{\mathcal{C}ol}$ (in which $u_{1}=0$) of $\underline{\{C=G_{2}\}}$ transversely in $\underline{\{C=G_{2}\}}$. The conclusion thus follows for small $\alpha$.
\end{proof}

At any intersection point $\tilde{p}_{0}$ of an invariant torus $\bar{A}$ with $\underline{\mathcal{C}ol}$, we have the direct sum decomposition 
$$T_{\tilde{p}_{0}} \Pi_{reg}^{\tilde{C}}=E^{9} \oplus E_{G_{2}, g_{2}},$$
in which $E^{9}$ is the 9-dimensional subspace tangent to $\{C=G_{2}=\tilde{C}, g_{2}=g_{2}(\tilde{p}_{0})\}$, and $E_{G_{2}, g_{2}}$ is the 2-dimensional subspace generated by $\dfrac{\partial}{\partial G_{2}}(\tilde{p}_{0})$ and $\dfrac{\partial}{\partial g_{2}}(\tilde{p}_{0})$. We observe that
\begin{itemize}
\item $E^{9} \subset T_{\tilde{p}_{0}} \underline{\mathcal{C}ol} + T_{\tilde{p}_{0}} \bar{A}=T_{\tilde{p}_{0}} \underline{\{C=G_{2}\}}$, and
\item $\dfrac{\partial}{\partial g_{2}}(\tilde{p}_{0}) \in T_{\tilde{p}_{0}} \bar{A} \cap T_{\tilde{p}_{0}} \underline{\mathcal{C}ol}$;
\end{itemize}
the first assertion comes from the transversality of $\bar{A}$ with $\underline{\mathcal{C}ol}$ in $\underline{\{C=G_{2}\}}$, while the second assertion holds since $G_{2}$ is a first integral of $\mathcal{F}_{Kep}+\overline{\mathcal{F}_{sec}^{n,n'}}$, {and $\bar{A}$ is obtained from an invariant torus in the reduced system by the symmetry of shifting $g_{2}$}.

The transformation $\psi^{3}$ is the time 1-map of a function $\hat{\mathcal{H}}$ satisfying the cohomological equation:
$$\nu_{quad, 2} \dfrac{\partial \hat{\mathcal{H}}}{\partial g_{2}}= \alpha \left(\mathcal{F}_{sec}^{1,3} -\dfrac{1}{2 \pi} \int_{0}^{2 \pi} \mathcal{F}_{sec}^{1,3} d \bar{g}_{2}\right);$$
recall that $\nu_{quad, 2}$ denotes the frequency of $g_{2}$ in the system $\mathcal{F}_{quad}$.

\begin{lem} \label{lem: torsion transformation}
There exists a small real number $\tilde{\varepsilon}>0$ independent of $\alpha$, and a non empty open subset $\underline{\mathcal{C}ol_{0}}$ of $\underline{\mathcal{C}ol}$  whose relative measure tends to $1$ locally in $\underline{\mathcal{C}ol}$ when $\tilde{\varepsilon} \to 0$, such that $\left|\Bigl.\dfrac{\partial^{2} \hat{\mathcal{H}}}{\partial g^{2}_{2}}\Bigr|_{\underline{\mathcal{C}ol_{0}}}\right|>2 \alpha \cdot \tilde{\varepsilon}$. 
\end{lem}
\begin{proof}It suffices to show that the function $\mathcal{F}_{sec}^{1,3}\Bigr|_{\underline{\mathcal{C}ol}}$ (and thus $\Bigl.\dfrac{\partial \hat{\mathcal{H}}}{\partial g_{2}}\Bigr|_{\underline{\mathcal{C}ol}}$) depends non-trivially on $g_{2}$. Indeed, this implies that the analytic function $\dfrac{1}{\alpha}\dfrac{\partial^{2} \hat{\mathcal{H}}}{\partial g^{2}_{2}}$ is not identically zero on $\underline{\mathcal{C}ol}$, therefore there exists $\tilde{\varepsilon}>0$ which bounds the absolute value of this function from below on an open set whose relative measure tends to $1$ locally in $\underline{\mathcal{C}ol}$ when $\tilde{\varepsilon} \to 0$.

To deduce that $\mathcal{F}_{sec}^{1,3}\Bigr|_{\underline{\mathcal{C}ol}}$ depends non-trivially on $g_{2}$, it is sufficient to observe from \cite{LaskarBoue} that when the two Keplerian ellipses are coplanar, 
$$\mathcal{F}_{sec}^{1,3}=-\dfrac{15}{64} \, \dfrac{(4 e_{1} +3 e_{1}^{3})e_{2}}{(1-e_{2}^{2})^{\frac{5}{2}}}\,\cos(g_{1}-g_{2}),$$
which depends non-trivially on $g_{2}$ when further restricted to $e_{1}=1$.
\end{proof}

We now determine the transformation $\phi^{3}$ more precisely: we require this transformation to preserve $C$. For this, we require $\hat{\mathcal{H}}$ to be invariant under rotations. Notice that the function $\nu_{quad, 2}$ is invariant under rotations. From \cite{LaskarBoue}, we see that on a dense open subset of $\Pi_{reg}$ where the angle $g_{2}$ is well-defined, the function $\mathcal{F}_{sec}^{1,3}(g_{2})$ is a linear combination of $\cos{g}_{2}$ and $\sin{g}_{2}$, with coefficients independent of $g_{2}$. {We may thus choose}
$$\hat{\mathcal{H}}=-\dfrac{\alpha}{\nu_{quad, 2}} \mathcal{F}_{sec}^{1,3}\left(g_{2}+\dfrac{\pi}{2}\right).$$ 

Let $\underline{\mathcal{C}ol'_{0}}=(\phi^{3})^{-1}(\underline{\mathcal{C}ol_{0}})$. This is an open subset of $\underline{\mathcal{C}ol'}=(\phi^{3})^{-1}(\underline{\mathcal{C}ol}) \subset \Pi_{reg}^{\tilde{C}}$ whose relative measure tends to $1$ locally in $\underline{\mathcal{C}ol}'$ when $\tilde{\varepsilon} \to 0$.

\begin{lem} For small enough $\alpha$, any invariant torus of $\mathcal{F}_{Kep}+\overline{\mathcal{F}_{sec}^{n,n'}}$ intersecting $\underline{\mathcal{C}ol'_{0}}$ is transverse to $\mathcal{C}ol'$ in $\Pi_{reg}$.  
\end{lem}
\begin{proof} Any $\tilde{p} \in \underline{\mathcal{C}ol'_{0}}$ can be written as $\tilde{p}=(\phi^{3})^{-1} (\tilde{p}_{0})$ for some $\tilde{p}_{0} \in \underline{\mathcal{C}ol_{0}}$. Let $\tilde{A}$ be the invariant torus which intersects $\underline{\mathcal{C}ol'_{0}}$ at $\tilde{p}$. {Since the transversality of $E^{9}$ and $E_{G_{2}, g_{2}}$ is independent of $\alpha$, for $\alpha$ sufficiently small,} we may decompose $T_{\tilde{p}} \Pi_{reg}^{\tilde{C}}$ as
$$T_{\tilde{p}_{0}} \Pi_{reg}^{\tilde{C}}= (\phi^{3})^{-1}_{*} E^{9} \oplus E'_{G_{2}, g_{2}},$$
in which $E'_{G_{2}, g_{2}}$ is the 2-dimensional space generated by $\dfrac{\partial}{\partial G_{2}}(\tilde{p})$ and $\dfrac{\partial}{\partial g_{2}}(\tilde{p})$. 
{We choose a basis $(\mathbf{e}_{1}, \cdots, \mathbf{e}_{9})$ of $(\phi^{3})^{-1}_{*} E^{9}$, and 9 vectors $(\mathbf{v}_{1}, \cdots, \mathbf{v}_{9})$ in $T_{\tilde{p}} \mathcal{C}ol'+T_{\tilde{p}} \tilde{A}$ such that $\mathbf{v}_{i}=\mathbf{e}_{i}+O(\alpha), i=1, \cdots, 9$. The vectors $\bigl(\dfrac{\partial}{\partial g_{2}}(\tilde{p}), \dfrac{\partial}{\partial G_{2}}(\tilde{p}), \mathbf{e}_{1}, \cdots, \mathbf{e}_{9}\bigr)$ thus forms a basis of $T_{\tilde{p}} \Pi_{reg}^{\tilde{C}}$. }

By Lemma \ref{lem: torsion transformation}, for $\alpha$ small enough, $\left|\dfrac{\partial^{2} \hat{\mathcal{H}}}{\partial \bar{g}^{2}_{2}}\right| > \alpha \cdot \tilde{\varepsilon}$ in an $O(\alpha)$-neighborhood of $\tilde{p}_{0}$ containing $\tilde{p}$. Hence we may write $(\phi^{3})^{-1}_{*} \dfrac{\partial}{\partial g_{2}} (\tilde{p}) \in T_{\tilde{p}} Col'$ as $(1+O(\alpha), \tilde{\alpha}, O(\alpha),\cdots, O(\alpha))$, in which $|\tilde{\alpha}|> \alpha \cdot \tilde{\varepsilon}$. 

In such a way, we have obtained 11 vectors {$\bigl(\dfrac{\partial}{\partial g_{2}}(\tilde{p}), \dfrac{\partial}{\partial G_{2}}(\tilde{p}), \mathbf{v}_{1}, \cdots, \mathbf{v}_{9}\bigr)$} in $T_{\tilde{p}} \underline{\mathcal{C}ol'}+T_{\tilde{p}} \tilde{A}$, which, written as row vectors, form a matrix of the form
$$
\begin{pmatrix}
1  & 0 & \vec{0}_{9}  \\
1+O(\alpha)& \tilde{\alpha} & O(\alpha)_{9} \\
O(\alpha)_{9}^{T} & O(\alpha)_{9}^{T} & Id_{9,9}+O(\alpha)_{9,9}
\end{pmatrix},
$$
in which $\vec{0}_{9}$ is the $1 \times 9$ zero matrix, $O(\alpha)_{9}$ (resp. $O(\alpha)_{9,9}$) is a $1 \times 9$ (resp. $9 \times 9$) matrix with only $O(\alpha)$ entries, and $Id_{9,9}$ is the $9 \times 9$ identity matrix. 

The determinant of this matrix is $\tilde{\alpha}+O(\alpha^{2})$, which is non-zero provided $\alpha$ is small enough. This implies $T_{\tilde{p}} \underline{\mathcal{C}ol'}+T_{\tilde{p}} \tilde{A}=T_{\tilde{p}} \Pi_{reg}^{\tilde{C}}$, \emph{i.e.} $\underline{\mathcal{C}ol'}$ is transverse to $\tilde{A}$ at $\tilde{p}$ in $\Pi_{reg}^{\tilde{C}}$.

The vector $\dfrac{\partial}{\partial G_{2}} (\tilde{p}_{0})$ being tangent to $\mathcal{C}ol$, the space $T_{\tilde{p}}\mathcal{C}ol'$ must contain a vector of the form $(O(\alpha),1+O(\alpha), O(\alpha)_{9})$. Since $\dfrac{\partial}{\partial G_{2}} (\tilde{p}_{0})$ is transverse to $T_{(\tilde{p}_{0})}\Pi_{reg}^{\tilde{C}}$ in $T_{(\tilde{p}_{0})}\Pi_{reg}$, any vector of the form $(O(\alpha),1+O(\alpha), O(\alpha)_{9})$ is also transverse to  $\Pi_{reg}^{\tilde{C}}$, provided $\alpha$ is small enough. Therefore $\tilde{A}$ is transverse to $\mathcal{C}ol'$ at $\tilde{p}$ in $\Pi_{reg}$.
\end{proof}

Since $(\phi^{3})^{-1}$ preserves $\mathcal{P}_{0}$ and $L_{2}$, it may only change the energy of a system at order $O(\alpha^{3})$. By hypothesis, The invariant tori intersecting $\mathcal{C}ol'$ transversely we have obtained have energy $O(\alpha^{3})$. We may then make proper $O(\alpha^{3})$-modifications of $\mathcal{L}_{1}$ to obtain an open set of invariant tori in the energy level $\mathcal{F}_{Kep}+\overline{\mathcal{F}_{sec}^{n, n'}}=0$  intersecting the set $\mathcal{C}ol'$ transversely in $\Pi_{reg}$. 

Therefore, those invariant 5-tori of $\mathcal{F}$ obtained in Subsection \ref{Subsection: Application of KAM theorem} intersecting $\mathcal{C}ol$ transversely form a set of positive measure in the energy level $\mathcal{F}=0$. Consequently, the intersection has codimension 3 in these 5-dimensional tori. If such a 5-dimensional torus is not ergodic, then it is foliated by 4-dimensional ergodic subtori obtained from one another by a rotation around $\vec{C}$. This gives a free SO(2)-action on the intersection of ${C}ol$ with the 5-dimensional tori, hence the intersection of ${\mathcal{C}ol}$ with each 4-dimensional ergodic torus is also of codimension 3.

\subsection{Conclusion}
\begin{lem} Let $\mathbf{K}$ be a submanifold of the $n$-dimensional torus $\T^n$ having codimension at least 2 in $\T^n$. Let $\tilde{\theta}=(\tilde{\theta}_1,\cdots,\tilde{\theta}_n)$ be the angular coordinates on $\T^n$; then almost all the orbits of the linear flow $\dfrac{d}{d t}\tilde{\theta}=\tilde{v},\,\tilde{v} \in \R^n$ do not intersect $\mathbf{K}$.
 \end{lem}
\begin{proof} By hypothesis, the set $\mathbf{K} \times \R \subset \T^{n} \times \R$ has Hausdorff dimension at most $n-1$. The set $\mathbf{K}'$ formed by orbits intersecting $\mathbf{K}$ is the image of $\mathbf{K} \times \R$ under the smooth mapping
$$\T^{n} \times \R \to \T^{n}\,\,\, (\tilde{\theta}(0), t) \mapsto \tilde{\theta}(t),$$
{in which $\tilde{\theta}(t)$ denotes the solution of this linear system with initial condition $\tilde{\theta}(0)$ when $t=0$}. Therefore $\mathbf{K}'$ has zero measure in $\T^{n}$.
\end{proof}

This lemma confirms that almost all trajectories on those ergodic tori (on which the flow is linear) intersecting $\mathcal{C}ol$ transversely does not intersect $\mathcal{C}ol$. Moreover, since the flow is irrational on these invariant ergodic tori of $\mathcal{F}$, almost all trajectories pass arbitrarily close to $\mathcal{C}ol$ without intersection. Each such trajectories give rise (via $K.S.$) to an collisionless orbit of $F$ which pass arbitrarily close to the set $\{Q_{1}=0\}$ of inner double collisions. Such orbits form a set of positive measure on the energy level $F=-f$. By varying $f$ and applying Fubini theorem, the collisionless orbits of $F$ along which the inner pair pass arbitrarily close to each other form a set of positive measure in the phase space $\Pi$. Theorem \ref{Theo: Main} is proved. 

\begin{rem}We have focused our attention on quasi-periodic almost-collision orbits along which the two instantaneous physical ellipses are almost coplanar. By analyticity of the system, the required non-degeneracy conditions, and therefore the result, can be improved to include more inclined cases as well.  
\end{rem}

\appendix
\section{Estimation of the perturbing function}
The following lemma is just a reformulation of \cite[Lemma~1.1]{QuasiMotionPlanar}. 

\begin{lem} \label{Legendre Expansion} When $\|Q_{1}\| \neq 0$, the expansion 
$$\mathcal{F}_{pert}=  K.S.^{*}\left(- \mu_1 m_2 \sum^{\phantom{ssss}}_{n \ge 2} \sigma_n P_n(\cos{\zeta}) \dfrac{\|Q_1\|^{n+1}}{\, \, \, \|Q_2\|^{n+1}}\right)= - \mu_1 m_2 \sum^{\phantom{ssss}}_{n \ge 2} \sigma_n\, K.S.^{*} \left(P_n(\cos{\zeta}) \dfrac{(1-e_{1} \cos u_{1})^{n+1}}{\, \, \, (1-e_{2} \cos u_{2})^{n+1}}\right) \alpha^{n+1}$$
 is convergent in $\dfrac{\|Q_1\|}{\,\|Q_2\|} \le \dfrac{1}{\hat{\sigma}},$ where 
 \begin{itemize}
\item $P_n$ is the n-th Legendre polynomial,
\item $\zeta$ is the angle between vectors $Q_{1}$ and $Q_{2}$,
\item $e_{1}$, $e_{2}$ are respectively the eccentricities of the two elliptic orbits, 
\item $u_{1}$, $u_{2}$ are respectively the eccentric anomalies of $Q_{1}$, $Q_{2}$ on their orbits, 
\item $\hat{\sigma}= max\{\sigma_0,\sigma_1\}$ and $ \sigma_n= \sigma_0^{n-1}+(-1)^n \sigma_1^{n-1}$.
 \end{itemize}
\end{lem}

We refer the notations and hypotheses of the next lemma to Subsection \ref{Subsection: Asynchronous elimination}.

\begin{lem} There exists a positive number $s>0$, such that $| \mathcal{F}_{pert} | \le \hbox{Cst}'\, |\alpha|^{3}$ in the $s$-neighborhood $T_{\mathcal{P}^{*},s}$ of $\mathcal{P}^{*}$ for some constant Cst independent of $\alpha$.
\end{lem}
\begin{proof} By continuity, there exists a positive number $s$, such that in a dense open set of $T_{\mathcal{P}^{*},s}$ defined by the condition $\|Q_{1}\| \neq 0$, we have
$$
|\cos{\zeta}| \le 2;\,\,
\left|\dfrac{\|Q_{1}\|}{\|Q_{2}\|} \right| \le \dfrac{4 |\alpha|}{1-e_{2}^{\wedge}}.
$$
in which $\cos{\zeta}$, $\|Q_{1}\|$ and $\|Q_{2}\|$ are considered as the corresponding analytically extensions of the original functions.

Using Bonnet's recursion formula of Legendre polynomials
$$(n+1) P_{n+1}(\cos{\zeta})=(2n+1) \, \cos{\zeta}\, P_{n}(\cos{\zeta})-n P_{n-1}(\cos{\zeta}), $$
by induction on $n$, we obtain $|P_{n}(\cos{\zeta})| \le 5^{n}$.

Thus 
\begin{align*} 
|\mathcal{F}_{pert}|&=  \mu_1 m_2 \left|\sum^{\phantom{ssss}}_{n \ge 2} \sigma_n P_n(\cos{\zeta}) \dfrac{\|Q_1\|^{n+1}}{\, \, \, \|Q_2\|^{n+1}}\right|\\
               & \le \mu_{1} m_{2} \sum^{\phantom{ssss}}_{n \ge 2} 5^{n} \left|\dfrac{\|Q_1\|}{ \, \|Q_2\|}\right|^{n+1}\\
               & \le \dfrac{\mu_{1} m_{2}}{5} \sum^{\phantom{ssss}}_{n \ge 2}  \dfrac{5^{n+1} 4^{n+1} |\alpha|^{n+1}}{\, \, \, (1-e_{2}^{\wedge})^{n+1}}\\
               & \le \dfrac{\mu_{1} m_{2}}{5}  \dfrac{20^{3} |\alpha|^{3}}{\, \, \, (1-e_{2}^{\wedge})^{2}} \dfrac{1}{1-e_{2}^{\wedge}- 20 |\alpha|}.
\end{align*}\normalsize
It is then sufficient to impose $\alpha \le \alpha^{\wedge}$ and make $s$ sufficiently small to ensure that $|\alpha| \le \dfrac{1-e_{2}^{\wedge}}{40}$. 

By continuity of the function $\mathcal{F}_{pert}$, the estimation holds in $T_{\mathcal{P}^{*},s}$.
\end{proof}

\section{Torsion of the Quadrupolar Tori} \label{Appendix: Torsion}
We fix $\vec{C}$ vertical. After Jacobi's elimination of node, we further normalize the coordinates $(G_{1}, g_{1}, G_{2})$ and parameters $C$ as in \cite{LidovZiglin} by setting 
$$\bbalpha=\dfrac{C}{L_1},\, \bbbeta=\dfrac{G_2}{L_1},\, \bbdelta=\dfrac{G_{1}}{L_{1}}, \,\bbomega=g_{1}.$$ 
In these coordinates, we have 
$$\mathcal{F}_{quad}=\dfrac{k}{\bbbeta^3}(\mathcal{W}+\dfrac{5}{3}),$$
in which $k$ is a irrelevant non-zero constant, and
\small
$$\mathcal{W}=-2 \bbdelta^2+\dfrac{(\bbalpha^2-\bbbeta^2-\bbdelta^2)^2}{4 \bbbeta^2} +5 (1-\bbdelta^2) \sin^2 \bbomega \, \left(\dfrac{(\bbalpha^2-\bbbeta^2-\bbdelta^2)^2}{4 \bbbeta^2 \bbdelta^2}-1\right). $$
\normalsize
Let $\overline{\mathcal{W}}(\bbdelta, \bbomega, \bbbeta;\bbalpha)=\dfrac{\mathcal{W}+\frac{5}{3}}{\bbbeta^3}$. This is a 2 degrees of freedom Hamiltonian in coordinates $(\bbdelta, \bbomega, \bbbeta, g_{2})$ with a parameter $\bbalpha$. We shall formulate our results in terms of $\overline{\mathcal{W}}$, from which the corresponding results for $\mathcal{F}_{quad}$ follow directly.

In the forthcoming proof, we deduce the existence of torsion of $\overline{\mathcal{W}}$ from a local approximating system $\overline{\mathcal{W}}'(\bbdelta, \bbomega, \bbbeta; \bbalpha)$ near $\{\bbdelta=\bbdelta_{min}:=|\bbalpha-\bbbeta|\}>0$  whose flow, for fixed $\bbbeta$, is linear in the $(\bbdelta, \bbomega)$-plane. Note that when $\bbbeta\neq \bbalpha$, the expression of $\overline{\mathcal{W}}$ is analytic at $\bbdelta=\bbdelta_{min}$. The local approximating system is thus obtained by developing $\overline{\mathcal{W}}$ into Taylor series of $\bbdelta$ at $\bbdelta=\bbdelta_{min}$. Finally, we show that the torsion does not vanish when $\bbalpha-\bbbeta \to 0$, which ensures the existence of torsion for quadrupolar tori at which $\bbalpha=\bbbeta$ close enough to the coplanar equilibrium with a degenerate inner ellipse. This is allowed since locally in this region, the symplectically reduced secular space by the $SO(3)$-symmetry is smooth. By doing so, we avoid choosing coordinates near these tori.

\begin{lem} \label{Appendix C, lem 1} The torsion of the quadrupolar tori near the lower boundary $\{\bbdelta=\bbdelta_{min}:=|\bbalpha-\bbbeta|\}>0$ exists and does not vanish when $\bbalpha-\bbbeta \to 0$. 
\end{lem}
\begin{proof}
We develop $\widetilde{\mathcal{W}}$ into Taylor series with respect to $\bbdelta$ at $\bbdelta=\bbdelta_{min}$. We set $\bbdelta_{1}=\bbdelta-\bbdelta_{min}$, and obtain
$$\widetilde{\mathcal{W}}=\bar{\Phi}(\bbalpha,\bbbeta) + \bar{\Xi}(\bbalpha,\bbbeta, \bbomega)\,\bbdelta_{1}+O (\bbdelta_{1}^{2}),$$
in which 
$$\bar{\Xi}(\bbalpha,\bbbeta, \bbomega)=-\dfrac{2 \left((9 \bbalpha^2 \bbbeta-6 \bbalpha \bbbeta^2+\bbbeta^3-4 \bbalpha^3+5 \bbalpha)+(-5 \bbalpha+5 \bbalpha^3-10 \bbalpha^2 \bbbeta+5 \bbalpha \bbbeta^2)\cos^{2} \omega\right)}{\bbbeta^4 |\bbalpha-\bbbeta|}.$$

We eliminate the dependence of $\bbomega$ in the linearized Hamiltonian $\bar{\Phi}(\bbalpha,\bbbeta)+\bar{\Xi}(\bbalpha,\bbbeta, \bbomega)\,\bbdelta_{1}$ by computing action-angle coordinates. The value of the action variable $\overline{\mathcal{I}}_{1}$ on the level curve $$E_{f}: \bar{\Phi}(\bbalpha,\bbbeta)+\bar{\Xi}(\bbalpha,\bbbeta, \bbomega)\,\bbdelta_{1}=f$$ is computed from the area between this curve and $\bbdelta_{1}=0$, that is
$$\overline{\mathcal{I}}_{1}=\dfrac{1}{2 \pi} \int_{E_{f}} \bbdelta_{1} d \bbomega=\dfrac{f-\bar{\Phi}(\bbalpha,\bbbeta)}{2 \pi} \int_{0}^{2 \pi} \dfrac{1}{\bar{\Xi}(\bbalpha,\bbbeta, \bbomega)} d \bbomega=\overline{\mathcal{I}}_{1}.$$

 We have then
$$\widetilde{\mathcal{W}}=\bar{\Phi}(\bbalpha,\bbbeta) + 2 \pi\left(\int_{0}^{2 \pi} \frac{1}{\bar{\Xi}(\bbalpha,\bbbeta, \bbomega)} d \bbomega\right)^{-1}\,\overline{\mathcal{I}}_{1}+O (\overline{\mathcal{I}}_{1}^{2}).$$

For $\overline{\mathcal{I}}_{1}$ small enough, the torsion of $\widetilde{\mathcal{W}}$ is dominated by the torsion of the term linear in $\overline{\mathcal{I}}_{1}$, which, represented by the absolute value of the determinant of the corresponding Hessian function with respect to $\overline{\mathcal{I}}_{1}$ and $\bbbeta$, is 

$$\left[2 \pi \frac{d}{d \bbbeta} \left(\int_{0}^{2 \pi} \frac{1}{\bar{\Xi}(\bbalpha,\bbbeta, \bbomega)} d \bbomega\right)^{-1}\right]^{2}.$$

Using the formula
$$\int_{0}^{2 \pi} \dfrac{d \bbomega}{a + b \cos \bbomega}=\dfrac{2 \pi}{\sqrt{a^{2}-b^{2}}}$$
we obtain
$$2 \pi\left(\int_{0}^{2 \pi} \frac{1}{\bar{\Xi}(\bbalpha,\bbbeta, \bbomega)} d \bbomega\right)^{-1}=-\dfrac{2\sqrt{\bbalpha+\bbbeta} \sqrt{9 \bbalpha^2 \bbbeta-6\bbalpha \bbbeta^2+\bbbeta^3-4 \bbalpha^3+5 \bbalpha}}{\bbbeta^4},$$ 

which depends non-trivially on $\bbbeta$. 

Moreover, at the limit $\bbalpha=\bbbeta$, the limit of the above function is $\dfrac{1125}{2 \bbbeta^{8}}$. By continuity, this proves the non-vanishing of the torsion for quadrupolar tori at which $\bbalpha=\bbbeta$ close enough to the coplanar equilibrium with a degenerate inner ellipse.
\end{proof}

\begin{ack} These results are part of my Ph.D. thesis \cite{ZLthesis} prepared at the Paris Observatory and the Paris Diderot University. Many thanks to my supervisors Alain Chenciner and Jacques Féjoz, for their enormous help and endless patience during these years.
\end{ack}

\bibliography{QuasiperiodicMotionSpatial}

\begin{thebibliography}{10}

\bibitem{ArnoldEncyclopedia}
V.I. Arnold, V.V. Kozlov, and A.I. Neishtadt.
\newblock {\em {M}athematical aspects of classical and celestial mechanics}.
\newblock Springer, 2006.

\bibitem{MarchalCollision}
C.~Marchal.
\newblock Collisions of stars by oscillating orbits of the second kind.
\newblock {\em Acta Astronautica}, 5(10):745--764, 1978.

\bibitem{ChencinerLlibre}
A.~Chenciner and J.~Llibre.
\newblock {A} note on the existence of invariant punctured tori in the planar
  circular restricted three-body problem.
\newblock {\em Ergodic theory and Dynamical Systems}, 8:63--72, 1988.

\bibitem{Fthesis}
J.~F{\'e}joz.
\newblock {\em Dynamique s{\'e}culaire globale du probl{\`e}me plan des trois
  corps et application {\`a} l'existence de mouvements quasip{\'e}riodiques}.
\newblock Th{\`e}se de l'université Paris 13, 1999.

\bibitem{QuasiMotionPlanar}
J.~F{\'e}joz.
\newblock Quasiperiodic motions in the planar three-body problem.
\newblock {\em Journal of Differential Equations}, 183(2):303--341, 2002.

\bibitem{LidovZiglin}
M.~Lidov and S.~Ziglin.
\newblock Non-restricted double-averaged three body problem in {H}ill's case.
\newblock {\em Celestial Mechanics and Dynamical Astronomy}, 13(4):471--489,
  1976.

\bibitem{SS}
E.~Stiefel and G.~Scheifele.
\newblock Linear and regular celestial mechnics.
\newblock {\em Die Grundlehren der mathematischen Wissenschaften, Berlin: J.
  Springer, 1971}, 1, 1971.

\bibitem{Mikkola}
S.~Mikkola.
\newblock A comparison of regularization methods for few-body interactions.

\bibitem{KSregularization}
L.~Zhao.
\newblock The {K}ustaanheimo-{S}tiefel regularization and the quadrupolar
  conjugacy.
\newblock {\em preprint}, 2013.

\bibitem{Waldvogel}
J.~Waldvogel.
\newblock Quaternions for regularizing celestial mechanics: the right way.
\newblock {\em Celestial Mechanics and Dynamical Astronomy}, 102(1):149--162,
  2008.

\bibitem{JefferysMoser}
W.H. Jefferys and J.~Moser.
\newblock Quasi-periodic solutions for the three-body problem.
\newblock {\em The Astronomical Journal}, 71:568, 1966.

\bibitem{LaskarBoue}
J.~Laskar and G.~Bou{\'e}.
\newblock Explicit expansion of the three-body disturbing function for
  arbitrary eccentricities and inclinations.
\newblock {\em Astronomy \& Astrophysics}, 522, 2010.

\bibitem{Pauli}
W.~Pauli.
\newblock {\"U}ber das {W}asserstoffspektrum vom {S}tandpunkt der neuen
  {Q}uantenmechanik.
\newblock {\em Zeitschrift f{\"u}r Physik A Hadrons and Nuclei},
  36(5):336--363, 1926.

\bibitem{Albouy}
A.~Albouy.
\newblock Lectures on the two-body problem.
\newblock {\em Classical and Celestial Mechanics, The Recife Lectures}, pages
  63--116, 2002.

\bibitem{FejozStability}
J.~F{\'e}joz.
\newblock D\'emonstration du th\'eor\`eme d'{A}rnold sur la stabilit\'e du
  syst\`eme plan\'etaire (d'apr\`es {H}erman)(revised version).
\newblock {\em Ergodic Theory and Dynamical Systems}, 24(5):1521--1582, 2004.

\bibitem{FejozMoser}
J.~F{\'e}joz.
\newblock The normal form of {M}oser and applications,.
\newblock {\em preprint}, 2013.

\bibitem{ZLthesis}
L.~Zhao.
\newblock {\em Solutions quasi-périodiques et solutions de quasi-collision du
  problème spatial des trois corps}.
\newblock Thèse de l'université Paris Diderot, 2013.

\end{thebibliography}
\end{document}